\documentclass[a4paper, 11pt, leqno]{amsart}

\usepackage{amsmath, amssymb, amsfonts, amsthm}
\usepackage{graphicx} 
\usepackage{mathtools}
\usepackage{xcolor}

\usepackage[T1]{fontenc}
\usepackage{geometry}
\usepackage{hyperref}
\usepackage{color}

\usepackage[shortlabels]{enumitem}  

\theoremstyle{plain}
\newtheorem{thm}{Theorem}[section]
\newtheorem{prop}[thm]{Proposition}
\newtheorem{lem}[thm]{Lemma}
\newtheorem{cor}[thm]{Corollary}

\theoremstyle{definition}

\theoremstyle{remark}
\newtheorem{rem}[thm]{Remark}

\title[Blow-up for energy supercritical defocusing NLS]{Blow-up for energy supercritical defocusing 
nonlinear Schr\"odinger equations in dimensions three, four and five}

\author{Kevin Buck, Alvaro Carballeira, Javier G\'omez-Serrano, Jia Shi}

\date{}

\begin{document}

\raggedbottom

\begin{abstract}
We prove finite time blow-up from smooth, radial, compactly supported initial data for the energy supercritical defocusing nonlinear Schr\"odinger equation $i\partial_tu+\Delta u-|u|^{p-1}u=0$ in $\mathbb R^d$, for $(d,p)\in\{(3,27),(4,9),(5,7)\}$. The solutions are asymptotically self-similar, their $L^\infty$ norm grows at the scaling rate $(T-t)^{-1(p-1)}$, and their critical Sobolev norm $\dot H^{s_c}$ diverges. This closes the conjecture of Bourgain \cite{Bourgain2000} on global well-posedness and scattering for the defocusing energy supercritical equation in dimensions three and four. The proof relies on a shooting argument and is computer-assisted.
\end{abstract}

\maketitle

\tableofcontents

\section{Introduction}

We consider the defocusing nonlinear Schr\"odinger equation
\begin{equation}\label{eq:nls}
	i\partial_tu+\Delta u-|u|^{p-1}u=0,
	\qquad (t,x)\in[0,T)\times\mathbb R^d.
\end{equation}
The scaling symmetry $u_\lambda(t,x)=\lambda^{\frac{2}{p-1}}u(\lambda^2t,\lambda x)$ leaves the homogeneous Sobolev norm $\dot H^{s_c}$ invariant, where
\begin{equation*}
	s_c=\frac d2-\frac{2}{p-1}.
\end{equation*}
The energy supercritical regime corresponds to $s_c>1$, or equivalently $p>1+\frac{4}{d-2}$ for $d\geq3$. In this range, the scaling critical regularity lies above the energy regularity, and conservation of mass and energy does not provide an a priori bound on the critical Sobolev norm. Throughout the paper, we restrict to $(d,p)\in\{(3,27),(4,9),(5,7)\}$, all in the energy supercritical regime. Our main result constructs finite-time, asymptotically self-similar blow-up solutions from smooth, radial, compactly supported initial data.

\begin{thm}\label{thm:blowup}
	For every $(d,p)\in\{(3,27),(4,9),(5,7)\}$, there are radial initial data
	$u_0\in C_c^\infty(\mathbb R^d;\mathbb C)$ whose corresponding solution
	of \eqref{eq:nls} blows up at some finite time $T>0$. The solution belongs to
	$C([0,T);H^m(\mathbb R^d;\mathbb C))$ for every integer $m\geq0$.
	Moreover, there are $\Omega\in\mathbb R$,
	$\mu>0$, and a smooth radial self-similar profile $Q$, such that, with
	$\tau=\log\frac{T}{T-t}$,
	\begin{equation*}
		\left\|(T-t)^{\frac{1}{p-1}}e^{i\Omega\tau}
		u\left(t,\sqrt{T-t}\,\cdot\right)-Q\right\|_{L^\infty}
		=O((T-t)^\mu)
		\quad\text{as }t\uparrow T.
	\end{equation*}
	Consequently,
	\begin{equation}\label{eq:blowup-rate}
		\|u(t)\|_{L^\infty}
		=\frac{\|Q\|_{L^\infty}+O((T-t)^\mu)}{(T-t)^{\frac{1}{p-1}}}
		\quad\text{as }t\uparrow T,
	\end{equation}
	and
	\begin{equation*}
		\lim_{t\uparrow T}\|u(t)\|_{\dot H^{s_c}}=\infty.
	\end{equation*}
\end{thm}

\begin{rem}
Merle, Rapha\"el, Rodnianski and Szeftel \cite{MerleRaphaelRodnianskiSzeftel2022} constructed smooth, well localized radial solutions blowing up in finite time for $(d,p)$ equal to $(5,9)$, $(6,5)$, $(8,3)$ or $(9,3)$. Theorem~\ref{thm:blowup} extends the existence of finite time blow-up to dimensions three and four, which were left open in their work, and to the nonlinearity $p=7$ in dimension five. The blow-up dynamics in dimension five are also different. Their construction is governed at leading order by a self-similar compressible Euler flow, with quantum pressure treated as a perturbation, and yields a blow-up rate which breaks the NLS scaling. In contrast, our solutions are asymptotically self-similar at the NLS scaling and obey the rate~\eqref{eq:blowup-rate}.
\end{rem}

\begin{rem}
The choice of the three pairs in Theorem~\ref{thm:blowup} is related to the computer-assisted profile construction. As explained in the proof of Lemma~\ref{lem:pure-slow-computer-bounds}, the profile ODE is numerically easier to solve for large $p$. Thus, we first locate a profile in that regime and then use a continuation method in $(d,p)$ to the desired pair. We selected $(3,27)$, $(4,9)$, and $(5,7)$ because this continuation and the subsequent estimates close with enough margin for these cases. Our nonrigorous shooting experiments find profiles for many other energy-supercritical parameters, although the rigorous construction in this paper is restricted to the stated cases.
\end{rem}

\subsection{Historical background}

The flow \eqref{eq:nls} conserves the mass and energy functionals
\begin{equation}
	\label{eq:energy}
	M(u) = \int_{\mathbb{R}^{d}}|u|^2\,dx, \qquad
	E(u) = \frac12\int_{\mathbb{R}^{d}}|\nabla u|^2\,dx
	+\frac1{p+1}\int_{\mathbb{R}^{d}}|u|^{p+1}\,dx,
\end{equation}
and it is invariant under the scaling $u_\lambda(t,x)=\lambda^{\frac{2}{p-1}}u(\lambda^2t,\lambda x)$, which preserves the homogeneous norm $\dot H^{s_c}$. The Cauchy problem is locally well-posed in $H^s$ for $s\geq s_c$ by Strichartz estimates \cite{GinibreVelo1979,CazenaveWeissler1990,Cazenave2003}, and the conservation laws yield global well-posedness in the mass subcritical and energy subcritical ranges. At the critical levels, the conserved quantities control precisely the scaling invariant norm, and the defocusing problem is by now completely understood: global well-posedness and scattering hold in the energy-critical case \cite{Bourgain1999,Grillakis2000,CollianderKeelStaffilaniTakaokaTao2008,RyckmanVisan2007,Visan2007} and in the mass-critical case \cite{TaoVisanZhang2007,KillipTaoVisan2009,Dodson2012,Dodson2016a,Dodson2016b}. In the energy supercritical regime $s_c>1$ considered here, no conservation law controls the critical norm, and the available results are conditional. Using the concentration compactness and rigidity method of Kenig--Merle \cite{KenigMerle2006,KenigMerle2008,KenigMerle2010,KenigMerle2011}, Killip--Visan \cite{KillipVisan2010} and Li--Li \cite{LiLi2022} proved that a solution which remains bounded in $\dot H^{s_c}$ is global and scatters; see \cite{Liu2026} for the most recent extension of the admissible exponents in high dimensions and \cite{Bulut2023} for blow-up criteria below the scaling regularity. In particular, the maximal solution issued from smooth data either is global or its $\dot H^{s_c}$ norm blows up in finite time. An analogous conditional scattering result for the defocusing $\dot H^{1/2}$-critical equation in dimensions $d\geq4$ was proved by Murphy \cite{Murphy2014}.

This collection of facts led to the belief that the a priori bound on the critical norm should always hold. Bourgain \cite{Bourgain2000} explicitly conjectured that the defocusing energy supercritical Schr\"odinger and wave equations are globally well-posed and scatter; see also the discussions in \cite{KenigMerle2011,KillipVisan2010}. Numerical simulations of Colliander--Simpson--Sulem \cite{CollianderSimpsonSulem2010} supported this belief. Global regularity was proved for logarithmically supercritical wave equations \cite{Tao2007,ColomboHaffter2023} and for a supercritical wave equation in two dimensions \cite{Struwe2011}, and large global solutions were constructed for supercritical wave equations in \cite{KriegerSchlag2017} and for the energy supercritical NLS in \cite{BeceanuDengSofferWu2019}. Indications in the opposite direction were the finite time blow-up constructed by Tao \cite{Tao2016,Tao2018} for defocusing wave and Schr\"odinger \emph{systems} sharing the same scaling and conservation laws as \eqref{eq:nls}, and the norm inflation results at low supercritical regularity of \cite{ChristCollianderTao2003,Thomann2008}.

The conjecture was disproved by Merle--Rapha\"el--Rodnianski--Szeftel \cite{MerleRaphaelRodnianskiSzeftel2022}, who constructed $C^\infty$ well localized radial initial data leading to finite time blow-up of \eqref{eq:nls} for $(d,p)$ equal to $(5,9)$, $(6,5)$, $(8,3)$ or $(9,3)$. Their mechanism is hydrodynamical. Through the Madelung transform $u=\sqrt\rho\,e^{i\psi}$ \cite{Madelung1927}, equation \eqref{eq:nls} becomes the compressible Euler system with adiabatic exponent $\gamma=\frac{p+1}{2}$, forced by a quantum pressure term proportional to $\nabla\big(\Delta\sqrt\rho/\sqrt\rho\big)$. In the blow-up regime of \cite{MerleRaphaelRodnianskiSzeftel2022}, the quantum pressure is a perturbation. The leading order dynamics are those of compressible Euler, and the blow-up profiles are the smooth radial imploding self-similar profiles constructed in the companion paper \cite{MerleRaphaelRodnianskiSzeftel2022I}; see also \cite{MerleRaphaelRodnianskiSzeftel2022II} for the corresponding singularity formation for the compressible Euler and Navier--Stokes equations. The Eulerian regime requires the self-similar blow-up speed $r$ to exceed $2$, which is only possible when $d\geq5$, and the existence of blow-up solutions in dimensions $d=3,4$ was left open \cite[Section~1.4]{MerleRaphaelRodnianskiSzeftel2022}. The solutions obtained in this way are of type II, meaning that they blow up at a rate different from the self-similar one and break scaling. Building on their non-radial implosion construction for compressible fluids \cite{CaoLaboraGomezSerranoShiStaffilani2026Euler}, Cao-Labora--G\'omez-Serrano--Shi--Staffilani \cite{CaoLabora2024} extended this construction to non-radial data for $(d,p)=(8,3)$, with smooth periodic initial data on $\mathbb T^8$ and with smooth initial data on $\mathbb R^8$ which do not vanish at infinity. The hydrodynamical formulation of \eqref{eq:nls} with the quantum pressure retained has also been studied on its own in the theory of quantum fluids; see Antonelli--Marcati \cite{AntonelliMarcati2009,AntonelliMarcati2012} and Carles--Danchin--Saut \cite{CarlesDanchinSaut2012} for finite energy weak solutions and for the precise relation with \eqref{eq:nls}.

Exact self-similar solutions $u(t,x)=(T-t)^{-\frac{1}{p-1}}e^{-i\Omega\tau}Q(x/\sqrt{T-t})$ are the basic blow-up objects of the focusing mass supercritical NLS. Their existence was observed numerically; see \cite{BuddChenRussell1999, PlechacSverak2001} and the book \cite{SulemSulem1999}, and rigorous constructions are recent: Bahri--Martel--Rapha\"el \cite{BahriMartelRaphael2021} in the slightly supercritical regime, Dahne--Figueras \cite{DahneFigueras2026} far from the critical exponent, including the three dimensional cubic case, Donninger--Sch\"orkhuber \cite{DonningerSchorkhuber2024} and Donninger--Lichtnecker \cite{DonningerLichtnecker2025} for the cubic and nearly cubic three dimensional equations, and Cao-Labora--G\'omez-Serrano--Shi--Staffilani \cite{CaoLaboraGomezSerranoShiStaffilani2025} for the hydrodynamic formulation. Their finite codimensional stability was established by Li \cite{Li2023}; see also \cite{Li2025a,Li2025b}. In the defocusing case, on the contrary, admissible self-similar solutions were expected not to exist \cite[Section~1.3]{MerleRaphaelRodnianskiSzeftel2022}. Self-similar solutions with finite energy are indeed excluded by energy conservation and scaling, see \eqref{eq:self-similar-energy} below. Theorem~\ref{thm:blowup} shows that, once the finite energy requirement on the profile is dropped, smooth radial self-similar profiles of the defocusing equation exist for $(d,p)\in\{(3,27),(4,9),(5,7)\}$, and that they govern the blow-up of smooth, compactly supported initial data at the self-similar rate \eqref{eq:blowup-rate}. In the profile equation \eqref{eq:complex-profile} the Laplacian and the nonlinearity enter at the same order, so the quantum pressure is not a perturbation and no compressible Euler regime is involved. To the best of our knowledge, this is the first construction of finite time blow-up for the defocusing NLS in dimensions three and four.

One of the fundamental pillars of the proof relies on a computer-assisted argument. The main paradigm is to replace floating-point numbers on a computer with rigorous bounds. These bounds are then propagated through every operation performed by the computer, taking into account any error made throughout the process. We refer the reader to the books \cite{Tucker2011,NakaoPlumWatanabe2019} and to the survey \cite{GomezSerrano2019} for a general treatment of computer-assisted proofs in PDE. In our case, the computer certifies that the solution regular at the origin and the solution with a pure slow tail at infinity can be matched at an intermediate radius, following the shooting and matching strategy of Dahne--Figueras \cite{DahneFigueras2026}. Closest to our setting, van den Berg--Groothedde--Williams \cite{vanDenBergGrootheddeWilliams2015} proved the existence of a localized radial solution of a singular Ginzburg--Landau ordinary differential equation. They rewrite the equation as an integral equation by means of a Green function, and prove with computer assistance that the resulting map is a contraction around a numerical approximation. Computer-assisted constructions of self-similar blow-up profiles have appeared in \cite{BuckmasterCaoLaboraGomezSerrano2023,BuckmasterCaoLaboraGomezSerrano2025} for compressible Euler, in \cite{ChenHouHuang2022,ChenHou2022,ChenHou2025} for the Hou--Luo model, the Boussinesq and the incompressible Euler equations, and in \cite{DahneFigueras2026,DonningerSchorkhuber2024} for the focusing NLS. Our implementation uses the Arb library \cite{Johansson2017} and the code is available at \cite{GITHUB}.

\subsection{Strategy of the proof}

We explain the main steps in the proof of Theorem~\ref{thm:blowup}. The construction consists of finding a smooth self-similar profile with prescribed behavior at infinity, and then producing solutions with localized initial data which converge to this profile in rescaled variables. The first step is computer-assisted and the second follows from spectral estimates and a finite-codimensional nonlinear stability argument.

We seek a radial profile $Q$ and a frequency $\Omega\in\mathbb R$ such that
\begin{equation*}
	u(t,x)=(T-t)^{-\frac{1}{p-1}}e^{-i\Omega\tau}Q(y),
	\qquad \tau=\log\frac{T}{T-t},\qquad y=\frac{x}{\sqrt{T-t}},
\end{equation*}
solves \eqref{eq:nls}. This gives the profile equation \eqref{eq:complex-profile}, in which dispersion and the nonlinearity both enter at leading order. After fixing the constant phase, regularity at the origin leaves two real shooting parameters, $Q(0)=A>0$ and $\Omega$, with $Q'(0)=0$. At infinity, the linearized equation has two distinct asymptotic modes,
\begin{equation*}
	Q_{\mathrm{slow}}(r)\sim C_s r^{-\frac{2}{p-1}+2i\Omega},
	\qquad
	Q_{\mathrm{fast}}(r)\sim C_f e^{-ir^2/4}
	r^{-d+\frac{2}{p-1}-2i\Omega}.
\end{equation*}
Since the second mode has a factor with a quadratic phase, each derivative introduces a growing factor. A nonzero fast mode has derivatives outside of $L^{2}$ at every integer order $j>s_c$, whereas the slow mode has derivatives in $L^{2}$ at those orders. We therefore impose $C_f=0$ and $C_s\ne0$, as required by the stability argument.

To construct such a profile, we match a solution regular at the origin to a solution with a pure slow tail at infinity, following the shooting approach of~\cite{DahneFigueras2026}. Near the origin, a finite Taylor expansion and a contraction argument give an enclosure of the exact solution. At infinity, we construct a finite asymptotic expansion and control its remainder again using a contraction argument. To do this we obtain a right inverse built from the two linear modes at infinity. This gives an exact solution at infinity with the prescribed slow coefficient and no fast component, together with an expansion to every algebraic order. At an intermediate radius, we match these two solutions. These are three real conditions in the three parameters $(A,\Omega,K)$, where $K=|C_s|^{p-1}$. A zero of the matching map joins the solutions up to a constant phase and yields the global profile of Theorem~\ref{thm:pure-slow-existence}.

The need for computer assistance arises in connecting the two endpoint constructions. For the fixed pairs $(d,p)$ in Theorem~\ref{thm:blowup}, our argument has no small parameter which controls the intermediate nonlinear regime. The local expansions alone therefore do not establish that the matching conditions can be satisfied. We overcome this difficulty by deriving explicit bounds for the endpoint remainders, the propagation of errors on finite intervals, and the parameter derivatives of the shooting problem. Interval arithmetic then verifies a finite collection of strict inequalities near an approximate matching zero. Brouwer's fixed-point theorem subsequently gives an exact zero. All computational assertions are collected in Lemma~\ref{lem:pure-slow-computer-bounds} and the code can be found in the supplementary material and at~\cite{GITHUB}.

The resulting profile has infinite energy because of its slow tail, so the exact self-similar solution does not yet give Theorem~\ref{thm:blowup}. To localize it, we use a finite-codimensional stability argument and write the rescaled solution as $V=Q+v$. We work in the radial space
\begin{equation*}
    \mathcal X_s
    =\dot H^s_{\mathrm{rad}}(\mathbb R^d;\mathbb C)
    \cap\dot H^{k_\star}_{\mathrm{rad}}(\mathbb R^d;\mathbb C),
    \qquad s_c<s<d/2,
    \qquad k_\star\in\mathbb N \text{ sufficiently large}.
\end{equation*}
This space contains $Q$, embeds into $L^\infty$, and is a Banach algebra. Moreover, cutting off the slow tail produces an arbitrarily small perturbation in $\mathcal X_s$. Our spectral analysis follows the strategy of obtaining dissipativity modulo a compact perturbation; see \cite{MerleRaphaelRodnianskiSzeftel2022, MerleRaphaelRodnianskiSzeftel2022II, BuckmasterCaoLaboraGomezSerrano2023,chen2026vorticity}. To apply complex spectral theory, we use the matrix formulation and the associated $\mathcal J$-invariance argument as in \cite{Schlag2009StableManifolds,Li2023}. The decay of the profile and its derivatives gives quadratic spatial decay for the potential coefficients of the linearized operator. An energy estimate, together with compact-perturbation theory, shows that the essential growth bound of the linearized semigroup is negative. After removing the finitely many neutral and unstable generalized eigenspaces, the linearized semigroup decays exponentially on the remaining stable subspace; see Proposition~\ref{prop:minrev-splitting}.

The nonlinear remainder is at least quadratic in $v$ in the $\mathcal X_s$ norm. We establish nonlinear stability using the standard Lyapunov--Perron construction; see \cite[Section 9.2]{Teschl2012}. A contraction argument selects the neutral and unstable components of the initial perturbation as functions of its stable component, so that $v$ decays exponentially in self-similar time. This defines a finite-codimensional manifold of initial data for $V$.

To construct localized data on this manifold, we use a finite-dimensional localization argument related to the constructions in \cite{LiZhou2025,Li2023}. We first replace $Q$ by $\chi_RQ$, where $\chi_R$ is a smooth radial cutoff at a large radius $R$. We then add a small correction $g_R$ from a fixed finite-dimensional space of smooth compactly supported radial functions. A further contraction chooses $g_R$ so that $V(0)=\chi_RQ+g_R$ belongs to this manifold. This proves Proposition~\ref{prop:global-stability} and gives smooth compactly supported initial data for the original equation.

\subsection*{Acknowledgements}
KB, AC and JGS were partially supported by the NSF under grants DMS-2245017, DMS-2554957 and DMS-2434314. JS was partially supported by the Travel Support for Mathematicians program from the Simons Foundation.

\subsection*{AI Disclosure}
The authors have used the LLMs GPT-5.6 and GPT-6 during the development of this work. These models have been employed for mathematical exploration, executing proof strategies and assisting with calculations. The main high level ideas of this paper were generated by the authors. The implementation of the code in the supplementary material was also done with the assistance of the same models. All mathematical statements, proofs and code scripts have been checked by the authors, who take full responsibility. All proofs and code scripts influenced by LLMs have been substantially rewritten.

\section{Self-similar profile equation}
\label{sec:profile-equation}

We seek exact self-similar solutions of \eqref{eq:nls} in the complex variable. Denote the nonlinearity by
\begin{equation*}
	N(q)=|q|^{p-1}q=q^{\frac{p+1}{2}}\overline q^{\,\frac{p-1}{2}}.
\end{equation*}
For $T>0$, use the similarity coordinates
\begin{equation}\label{eq:coordinates}
	\tau=\log\frac{T}{T-t},\qquad y=\frac{x}{\sqrt{T-t}}.
\end{equation}
Thus $d\tau/dt=(T-t)^{-1}$ and $t\uparrow T$ corresponds to $\tau\to\infty$. For $\Omega\in\mathbb R$, set
\begin{equation}\label{eq:complex-ansatz}
	u(t,x)=(T-t)^{-\frac{1}{p-1}}e^{-i\Omega\tau}Q(y),
	\qquad y=\frac{x}{\sqrt{T-t}}.
\end{equation}
Since $\partial_t\tau=(T-t)^{-1}$, $\partial_ty=y/[2(T-t)]$, substitution into \eqref{eq:nls} gives the profile equation
\begin{equation}\label{eq:complex-profile}
	\Delta_yQ-|Q|^{p-1}Q+\Omega Q
	+i\left(\frac{1}{p-1}Q+\frac12 y\cdot\nabla_yQ\right)=0.
\end{equation}
If a nonzero bounded smooth profile $Q$ exists globally, \eqref{eq:complex-ansatz} gives an exact solution with
\begin{equation*}
	\|u(t)\|_{L^\infty}
	=(T-t)^{-\frac{1}{p-1}}\|Q\|_{L^\infty}.
\end{equation*}
However, these solutions may not have finite energy \eqref{eq:energy}. For the exact ansatz \eqref{eq:complex-ansatz}, scaling gives
\begin{equation}\label{eq:self-similar-energy}
	E(u(t))=(T-t)^{s_c-1}E(Q).
\end{equation}
If $E(Q)<\infty$ and the associated solution lies in a class where energy is conserved, then $s_c>1$ forces $E(Q)=0$ and thus $Q=0$. Thus the nontrivial profiles must be localized to prove Theorem~\ref{thm:blowup}, as in~\cite{MerleRaphaelRodnianskiSzeftel2022}.

We now restrict to radial self-similar profiles and study their asymptotic behavior at the origin and at infinity. For a radial profile, write $Q(y)=Q(r)$ with $r=|y|$. Equation \eqref{eq:complex-profile} becomes
\begin{equation}\label{eq:complex-profile-radial}
	Q''+\frac{d-1}rQ'-|Q|^{p-1}Q+\Omega Q
	+i\left(\frac{1}{p-1}Q+\frac r2Q'\right)=0,
	\qquad r>0.
\end{equation}
After fixing the constant phase, we prescribe
\begin{equation}\label{eq:complex-profile-origin}
	Q(0)=A>0,\qquad Q'(0)=0.
\end{equation}
Smooth radial profiles extend evenly across $r=0$. Write $Q(r)=A+Br^2+O(r^4)$, so that
\begin{equation*}
	Q''(r)+\frac{d-1}rQ'(r)=2dB+O(r^2),
	\qquad \frac r2Q'(r)=O(r^2).
\end{equation*}
The coefficients in \eqref{eq:complex-profile-radial} therefore give $2dB-A^p+\left(\Omega+\frac{i}{p-1}\right)A=0$, and hence
\begin{equation*}
	Q(r)=A+\frac{A}{2d}\left(A^{p-1}-\Omega-\frac{i}{p-1}\right)r^2
	+O(r^4).
\end{equation*}

At infinity, the two leading behaviors are determined by the linear equation
\begin{equation*}
	q''+\left(\frac{d-1}r+\frac{ir}{2}\right)q'
	+\left(\Omega+\frac{i}{p-1}\right)q=0.
\end{equation*}
We first study this asymptotic regime formally. Exact linear solutions with these asymptotics are constructed in Subsection~\ref{subsec:profile-linear-modes}. For an algebraic ansatz $q(r)=r^\alpha$, substitution gives
\begin{equation*}
	q''+\left(\frac{d-1}r+\frac{ir}{2}\right)q'
	+\left(\Omega+\frac{i}{p-1}\right)q
	=r^\alpha\left[
		\Omega+\frac{i}{p-1}+\frac{i\alpha}{2}
		+\frac{\alpha(\alpha+d-2)}{r^2}\right].
\end{equation*}
Thus the leading coefficient vanishes precisely when $\alpha=-\frac{2}{p-1}+2i\Omega$. This balance is consistent with the nonlinear equation. Thus $q(r)=C_s r^{-2/(p-1)+2i\Omega}$, and the nonlinearity becomes
\begin{equation*}
	|q|^{p-1}q=|C_s|^{p-1}r^{-2}q.
\end{equation*}
Both the radial Laplacian and the nonlinearity are therefore of order $r^{-2}$ relative to $q$.

The second behavior includes a quadratic phase oscillatory factor. Substituting $q(r)=e^{-ir^2/4}r^\beta$, for which
\begin{equation*}
	\frac{q'}q=-\frac{ir}{2}+\frac\beta r,
	\qquad
	\frac{q''}q=-\frac{r^2}{4}-i\left(\beta+\frac12\right)
	+\frac{\beta(\beta-1)}{r^2},
\end{equation*}
gives
\begin{equation*}
	q''+\left(\frac{d-1}r+\frac{ir}{2}\right)q'
	+\left(\Omega+\frac{i}{p-1}\right)q
	=q\left[\Omega+i\left(\frac{1}{p-1}-\frac d2-\frac\beta2\right)
		+\frac{\beta(\beta+d-2)}{r^2}\right].
\end{equation*}
The terms of order $r^2q$ cancel, and the remaining leading coefficient vanishes when $\beta=-d+\frac{2}{p-1}-2i\Omega$. For this ansatz, $|q|^{p-1}=r^{(-d+2/(p-1))(p-1)}$, so the nonlinearity is again lower order. The formal balances yield the following tail modes, which we refer to as slow and fast modes according to their leading terms,
\begin{equation*}
	\begin{aligned}
		Q_{\mathrm{slow}}(r) & \sim C_s r^{-\frac{2}{p-1}+2i\Omega},                      \\
		Q_{\mathrm{fast}}(r) & \sim C_f e^{-\frac{ir^2}{4}}r^{-d+\frac{2}{p-1}-2i\Omega}.
	\end{aligned}
\end{equation*}

The fast oscillation limits derivative integrability. For a nonzero fast mode and an integer $j\geq0$, formally differentiating its asymptotic expansion gives
\begin{equation*}
	\frac{d^j}{dr^j}Q_{\mathrm{fast}}(r)
	=C_f\left(-\frac i2\right)^j
	e^{-\frac{ir^2}{4}}r^{-d+\frac{2}{p-1}-2i\Omega+j}
	\left(1+O(r^{-2})\right).
\end{equation*}
Their squared integrability at infinity is equivalent to convergence of
\begin{equation*}
	\int_1^\infty r^{2(-d+\frac{2}{p-1}+j)}r^{d-1}\,dr
	=\int_1^\infty r^{2j-d-1+\frac{4}{p-1}}\,dr.
\end{equation*}
This integral converges exactly when $j<d/2-2/(p-1)=s_c$. Thus a nonzero fast mode has derivatives not in $L^{2}$ at every integer order $j>s_c$. Derivatives of the slow mode instead decay like $r^{-2/(p-1)-j}$ and are square-integrable exactly when $j>s_c$. Hence the slow mode cannot cancel the failure of integrability of the fast mode. However, the nonzero slow tail is not in $L^2$, nor are its first derivatives in $L^2$. Our approach is to remove the fast mode and then localize the slow mode to obtain initial data in every $H^m$.

This section focuses on the construction of a profile satisfying
\begin{equation*}
	C_f=0,\qquad C_s\ne0,
\end{equation*}
while the localization of the profile will be studied in Section~\ref{sec:self-similar-blowup}.

For $A>0$ and $\Omega\in\mathbb R$, let $Q_{A,\Omega}$ denote the unique maximal regular solution of \eqref{eq:complex-profile-radial}--\eqref{eq:complex-profile-origin}. Note that equation~\eqref{eq:complex-profile-radial} is singular at the origin. The following argument shows that this equation defines a unique smooth radial solution starting at the origin, and that this solution cannot vanish while it exists. This will be used later for the construction matching an interior solution to a slow tail at infinity.

\begin{lem}\label{lem:profile-local-regularity}
	For each $A>0$ and $\Omega\in\mathbb R$, the profile problem has a
	unique maximal smooth radial solution $Q_{A,\Omega}$ on $[0,R_*)$.
	It never vanishes, and if $R_*<\infty$, then
	$\sup_{r<R_*}|Q_{A,\Omega}(r)|=\infty$.
	The solution depends continuously on $(A,\Omega)$ in $C^1$ on
	compact intervals of existence.
\end{lem}

\begin{proof}
	Multiplication of \eqref{eq:complex-profile-radial} by an integrating factor $r^{d-1}e^{ir^2/4}$ gives
	\begin{equation*}
		\bigl(r^{d-1}e^{\frac{ir^2}{4}}Q'(r)\bigr)'
		=r^{d-1}e^{\frac{ir^2}{4}}\left(N(Q(r))-\left(\Omega+\frac{i}{p-1}\right)Q(r)\right).
	\end{equation*}
	For a regular solution, the expression inside the derivative tends to zero at the origin. Integration therefore yields
	\begin{equation}\label{eq:profile-local-integral}
		Q'(r)=r^{1-d}e^{-\frac{ir^2}{4}}\int_0^r
		\sigma^{d-1}e^{\frac{i\sigma^2}{4}}\left(N(Q(\sigma))-\left(\Omega+\frac{i}{p-1}\right)Q(\sigma)\right)\,d\sigma.
	\end{equation}
	We construct a solution by integrating this identity once more and a fixed point argument. For $R>0$, consider the closed ball $\mathcal B_R=\{q\in C([0,R];\mathbb C):\|q-A\|_\infty\leq1\}$ and the map
	\begin{equation*}
		(\mathcal Tq)(r)=A+\int_0^r t^{1-d}e^{-\frac{it^2}{4}}
		\int_0^t \sigma^{d-1}e^{\frac{i\sigma^2}{4}}
		\left(N(q(\sigma))-\left(\Omega+\frac{i}{p-1}\right)q(\sigma)\right)\,d\sigma\,dt.
	\end{equation*}
	The integrand in the outer integral extends continuously to zero. Since $N$ is polynomial in the real and imaginary parts, there are constants $M,L>0$ such that, whenever $|z-A|,|w-A|\leq1$,
	\begin{equation*}
		\left|N(z)-\left(\Omega+\frac{i}{p-1}\right)z\right|\leq M,
		\qquad
		\left|N(z)-N(w)-\left(\Omega+\frac{i}{p-1}\right)(z-w)\right|\leq L|z-w|.
	\end{equation*}
	For $q\in\mathcal B_R$ and $0\leq r\leq R$, the exponential factors have modulus one, so
	\begin{equation*}
		\begin{aligned}
			|(\mathcal Tq)(r)-A|
			 & \leq\int_0^r t^{1-d}\int_0^t
			\sigma^{d-1}\left|N(q(\sigma))-\left(\Omega+\frac{i}{p-1}\right)q(\sigma)\right|\,d\sigma\,dt \\
			 & \leq M\int_0^r t^{1-d}\left(\int_0^t \sigma^{d-1}\,d\sigma\right)dt
			=\frac Md\int_0^r t\,dt
			=\frac{Mr^2}{2d}.
		\end{aligned}
	\end{equation*}
	Similarly, for $q,\widetilde q\in\mathcal B_R$,
	\begin{equation*}
		\begin{aligned}
			 & |(\mathcal Tq)(r)-(\mathcal T\widetilde q)(r)|                                     \\
			 & \qquad \leq\int_0^r t^{1-d}\int_0^t \sigma^{d-1}
			\bigl|N(q(\sigma))-N(\widetilde q(\sigma))
			-\left(\Omega+\frac{i}{p-1}\right)(q(\sigma)-\widetilde q(\sigma))\bigr|\,d\sigma\,dt \\
			 & \qquad \leq L\int_0^r t^{1-d}\int_0^t
			\sigma^{d-1}|q(\sigma)-\widetilde q(\sigma)|\,d\sigma\,dt
			=\frac{Lr^2}{2d}\|q-\widetilde q\|_\infty.
		\end{aligned}
	\end{equation*}
	Taking the supremum over $r\in[0,R]$ gives
	\begin{equation*}
		\|\mathcal Tq-A\|_\infty\leq\frac{MR^2}{2d},
		\qquad
		\|\mathcal Tq-\mathcal T\widetilde q\|_\infty
		\leq\frac{LR^2}{2d}\|q-\widetilde q\|_\infty.
	\end{equation*}
	Choose $R$ so that $MR^2/(2d)\leq1$ and $LR^2/(2d)<1$. The contraction mapping theorem gives a unique fixed point. It satisfies \eqref{eq:profile-local-integral}, with $|Q'(r)|\leq Mr/d$, and hence is $C^1$ up to zero with $Q(0)=A$ and $Q'(0)=0$. Differentiating the integral formula for $r>0$ shows that it solves the profile equation there. Every regular solution with these initial data lies in $\mathcal B_R$ on a sufficiently short interval and satisfies the same fixed-point equation. This proves local uniqueness.

	To verify smoothness at the origin, substitute $\sigma=r\theta$ in \eqref{eq:profile-local-integral}:
	\begin{equation*}
		Q'(r)=r\int_0^1\theta^{d-1}e^{\frac{ir^2(\theta^2-1)}{4}}
		\left(N(Q(r\theta))-\left(\Omega+\frac{i}{p-1}\right)Q(r\theta)\right)\,d\theta.
	\end{equation*}
	The even extension $Q(-r)=Q(r)$ is $C^1$, since $Q'(0)=0$, and satisfies this identity for negative $r$ as well. If this extension is $C^j$, differentiation under the integral on the fixed interval $[0,1]$ shows that the right-hand side is $C^j$. Thus $Q$ is $C^{j+1}$, and induction proves smoothness and evenness. A smooth even function of $r$ defines a smooth radial function of $x$.

	Away from zero the equation is a smooth first-order system for $(Q,Q')$, so local uniqueness gives a unique maximal continuation on $[0,R_*)$. If $R_*<\infty$ and $Q$ is bounded there, then $N(Q)-\left(\Omega+\frac{i}{p-1}\right)Q$ is bounded and \eqref{eq:profile-local-integral} bounds $Q'$ on $[0,R_*)$. The differential equation then bounds $Q''$ on $[R_*/2,R_*)$. Consequently $Q$ and $Q'$ have finite limits at $R_*$, and local existence for the first-order system extends the solution past $R_*$. This contradiction proves the continuation criterion. The contraction construction also gives continuous dependence on the parameters.

	Finally, set $\rho=|Q|^2$ and $J=2\operatorname{Im}(\overline Q Q')$. Taking twice the imaginary part of \eqref{eq:complex-profile-radial} multiplied by $\overline Q$ gives, for $r>0$,
	\begin{equation*}
		J'+\frac{d-1}rJ+\frac{2}{p-1}\rho+\frac r2\rho'=0.
	\end{equation*}
	Here the nonlinear term and $\Omega|Q|^2$ are real, and $\rho'=2\operatorname{Re}(\overline Q Q')$. It follows that
	\begin{equation*}
		\frac{d}{dr}\left[r^{d-1}\left(J+\frac r2\rho\right)\right]
		=s_c r^{d-1}\rho.
	\end{equation*}
	Since $Q$ is bounded and $Q'(r)=O(r)$ at zero, the expression in brackets tends to zero there. Integration yields
	\begin{equation}\label{eq:profile-current}
		r^{d-1}\left(2\operatorname{Im}(\overline Q Q')+\frac r2|Q|^2\right)
		=s_c\int_0^r \sigma^{d-1}|Q(\sigma)|^2\,d\sigma.
	\end{equation}
	Because $s_c>0$ and $Q(0)=A>0$, continuity makes the right-hand side strictly positive for every $r\in(0,R_*)$. If $Q(r)=0$, the left-hand side would vanish. Hence $Q$ has no zeros on $[0,R_*)$.
\end{proof}

\subsection{Linear exterior modes}
\label{subsec:profile-linear-modes}

The linear part of the radial profile equation \eqref{eq:complex-profile-radial} is
\begin{equation}\label{eq:profile-linear-operator}
	L_\Omega q=q''+\left(\frac{d-1}r+\frac{ir}{2}\right)q'
	+\left(\Omega+\frac{i}{p-1}\right)q,\qquad r>0.
\end{equation}
Set
\begin{equation}\label{eq:mu-and-kappa}
	\mu=\frac{1}{p-1}-i\Omega, \quad \kappa=\frac d2-\mu,
\end{equation}
so that $\Omega+\frac{i}{p-1}=i\mu$, $\operatorname{Re}\mu=\frac{1}{p-1}>0$, and $\operatorname{Re}\kappa=\frac d2-\frac{1}{p-1}>0$. The following lemma constructs a fundamental system of solutions of the linear equation $L_\Omega q=0$ on $(0,\infty)$, with the previously stated asymptotics at infinity.

\begin{lem}\label{lem:profile-linear-modes}
	The equation $L_\Omega q=0$ on $(0,\infty)$ has solutions given by
	\begin{equation}\label{eq:profile-general-modes}
		\begin{aligned}
			S_\Omega(r) & =\frac{4^{-\mu}}{\Gamma(\mu)}
			\int_0^\infty e^{-\frac{r^2t}{4}}t^{\mu-1}(1+it)^{\frac{d}{2}-1-\mu}\,dt, \\
			F_\Omega(r) & =\frac{4^{-\kappa}e^{-\frac{ir^2}{4}}}{\Gamma(\kappa)}
			\int_0^\infty e^{-\frac{r^2t}{4}}t^{\kappa-1}(1-it)^{\frac{d}{2}-1-\kappa}\,dt.
		\end{aligned}
	\end{equation}
	Both integrals converge absolutely for every $r>0$.
\end{lem}

\begin{proof}
	The change of variables $x=r^2/4$, with $q(r)=g_\mu(x)$, transforms $L_\Omega q=0$ into
	\begin{equation*}
		xg_\mu''+(d/2+ix)g_\mu'+i\mu g_\mu=0.
	\end{equation*}
	To exploit the affine dependence of the coefficients on $x$, seek a Laplace integral representation
	\begin{equation*}
		g_\mu(x)=\int_0^\infty e^{-xt}h_\mu(t)\,dt.
	\end{equation*}
	Substituting this ansatz and integrating the terms containing $x$ by parts formally gives
	\begin{equation*}
		xg_\mu''+(d/2+ix)g_\mu'+i\mu g_\mu
		=\int_0^\infty e^{-xt}
		\bigl(t(t-i)h_\mu'+\bigl((2-d/2)t+i(\mu-1)\bigr)h_\mu\bigr)\,dt,
	\end{equation*}
	provided the boundary term $[e^{-xt}t(t-i)h_\mu(t)]_0^\infty$ vanishes. Thus it suffices to solve the first-order equation
	\begin{equation*}
		t(t-i)h_\mu'+\bigl((2-d/2)t+i(\mu-1)\bigr)h_\mu=0,
		\quad\text{equivalently}\quad
		\frac{h_\mu'}{h_\mu}=\frac{\mu-1}{t}
		+\frac{i(d/2-1-\mu)}{1+it}.
	\end{equation*}
	Integration yields, up to a multiplicative constant,
	\begin{equation*}
		h_\mu(t)=t^{\mu-1}(1+it)^{\frac{d}{2}-1-\mu}.
	\end{equation*}
	This function is $O(t^{1/(p-1)-1})$ at zero and $O(t^{d/2-2})$ at infinity. Consequently its Laplace integral converges absolutely for $x>0$, can be differentiated under the integral sign to every order, and has the required vanishing boundary term. The preceding calculation is therefore justified. Multiplication by $4^{-\mu}/\Gamma(\mu)$ gives the stated formula for $S_\Omega$.

	To obtain the second solution, extract the quadratic oscillatory factor by seeking $q(r)=e^{-ix}g_\mu(x)$, again with $x=r^2/4$. Direct substitution gives
	\begin{equation*}
		L_\Omega q=e^{-ix}
		\bigl(xg_\mu''+(d/2-ix)g_\mu'-i\kappa g_\mu\bigr),
	\end{equation*}
	since $\mu+\kappa=d/2$. The equation in parentheses is the preceding transformed equation with $i$ replaced by $-i$ and $\mu$ by $\kappa$. The same Laplace integral construction therefore gives the function $t^{\kappa-1}(1-it)^{d/2-1-\kappa}$. It is integrable at zero because $\operatorname{Re}\kappa>0$ and has polynomial growth at infinity, so the same convergence and integration-by-parts arguments apply. Multiplication by $4^{-\kappa}/\Gamma(\kappa)$ and restoration of the factor $e^{-ix}$ yield $F_\Omega$. These representations are consistent with the standard confluent hypergeometric integral formulas; see~\cite[Sections~13.2 and~13.4]{OldeDaalhuis2026}.

	Changing variables $t\mapsto4t/r^2$ gives
	\begin{equation*}
		\begin{aligned}
			S_\Omega(r) & =r^{-2\mu}\left(1+\frac{4i\mu(d/2-1-\mu)}{r^2}
			+O_{\Omega}(r^{-4})\right),                                  \\
			F_\Omega(r) & =e^{-\frac{ir^2}{4}}r^{-2\kappa}
			\left(1+O_{\Omega}(r^{-2})\right).
		\end{aligned}
	\end{equation*}
	The same rescaling applied to the differentiated integrals justifies differentiation of these asymptotic expansions. Abel's identity and the leading asymptotics then give
	\begin{equation}\label{eq:profile-Wronskian}
		W_\Omega=S_\Omega F_\Omega'-S_\Omega'F_\Omega
		=-\frac i2r^{1-d}e^{-\frac{ir^2}{4}}.
	\end{equation}
	Thus $W_\Omega(r)\ne0$ for every $r>0$, proving that $S_\Omega$ and $F_\Omega$ form a fundamental system.
\end{proof}

For a nonlinear profile with a nonzero slow tail, $C_s$ denotes the leading slow coefficient. We define the fast coefficient $C_f$ by requiring that there exist constants $c_1,c_2\in\mathbb{C}$ with
\begin{equation}\label{eq:fast-coefficient-definition}
	Q(r)=C_sr^\beta\bigl(1+c_1r^{-2}+c_2r^{-4}\bigr)
	+C_fF_\Omega(r)+o(r^{-d+\frac{2}{p-1}}),
	\qquad \beta=-\frac{2}{p-1}+2i\Omega.
\end{equation}
Note that here $F_\Omega$ carries the factor $e^{-ir^2/4}$. Comparing the distinct decay rates in \eqref{eq:fast-coefficient-definition} determines $C_s$ and $C_f$ uniquely. For our parameter pairs, $4/(p-1)<d<6+4/(p-1)$, so the fast mode decays faster than the leading slow mode, while $r^{\operatorname{Re}\beta-6}=o(r^{-d+2/(p-1)})$.

The exterior construction below proves this expansion with $C_f=0$. Considering subsequent nonlinear slow terms is necessary in this definition since, in the case of $(d,p)=(5,7)$, the first term of $F_\Omega$ is of order $O(r^{-14/3})$, while the first two nonlinear slow terms are of orders $O(r^{\beta-2})=O(r^{-7/3})$ and $O(r^{\beta-4})=O(r^{-13/3})$, respectively.

\section{Construction of the self-similar profile}

To construct a profile with $C_f=0$ and $C_s\ne0$, we first construct the slow exterior solution and obtain rigorous bounds on its value and derivative at a finite radius. We then match it to the regular-origin solution by Brouwer's fixed-point theorem. All statements established by computation are collected in Lemma~\ref{lem:pure-slow-computer-bounds}. The code is available in the supplementary material and at~\cite{GITHUB}. The estimates that justify the computation are proved below.

The computer-assisted shooting and matching strategy used in this subsection is inspired by the work in~\cite{DahneFigueras2026}. As in their construction, we build solutions satisfying the prescribed conditions at the origin and at infinity and certify that they match at an intermediate radius. Related computer-assisted constructions have appeared in many works, for example,~\cite{vanDenBerg2011,LessardMirelesJamesReinhardt2014}.

Use the parameters
\begin{equation*}
	\vartheta=(A,\Omega,K),\qquad A>0,\quad K>0,
\end{equation*}
where $A$ is the value of the profile at the origin, $\Omega$ is the self-similar frequency, and $K = |C_s|^{p-1}$ relates to the leading coefficient of the slow tail at infinity. We introduce the variable $K$ because it naturally enters the nonlinear expansion at infinity. Let $\vartheta_0=(A_0,\Omega_0,K_0)$ be the parameters supplied by the computer-assisted code. These parameters were found using a shooting and continuation method to discover approximate solutions of the profile equation. They can be found in the accompanying code. Their approximate values are
\begin{equation*}
	\begin{array}{c|rrr}
		(d,p)  & A_0                & \Omega_0            & K_0               \\ \hline
		(3,27) & 0.0501485544258024 & -1.3966245573341954 & 70.8930132091120  \\
		(4,9)  & 0.0014126481622802 & -2.9442997039706108 & 230.5787044287510 \\
		(5,7)  & 0.0409508188010580 & -1.7226834479870780 & 120.4293553317432
	\end{array}
\end{equation*}
Define the parameter cube
\begin{equation}\label{eq:pure-slow-parameter-box}
	D=\{\vartheta\in\mathbb R^3:
	\|\vartheta-\vartheta_0\|_\infty\leq\rho\},\qquad
	\rho=\begin{cases}
		2^{-240}, & (d,p)=(3,27), \\
		2^{-320}, & (d,p)=(4,9),  \\
		2^{-220}, & (d,p)=(5,7).
	\end{cases}
\end{equation}
This small parameter cube is needed in the computer-assisted proof to control the propagation of the error bound. Our shooting method is extremely sensitive to the parameters, and the smallness of $\rho$, together with the use of multiple-precision arithmetic, is necessary to obtain a rigorous enclosure of the solution. The next theorem shows that the parameters $(A_0,\Omega_0,K_0)$ can be perturbed inside the parameter cube $D$ to yield a global regular profile with a slow tail at infinity.

\begin{thm}
	\label{thm:pure-slow-existence}
	There is a $\vartheta_*=(A_*,\Omega_*,K_*)\in D$ such that
	the regular profile $Q_{A_*,\Omega_*}$ is global and has
	\begin{equation*}
		C_f=0,\qquad |C_s|=K_*^{\frac{1}{p-1}}>0.
	\end{equation*}
	In particular, the profile is smooth, regular at the origin, and nonvanishing at every finite radius. At infinity it has a slow expansion to every algebraic order, as specified in Corollary~\ref{cor:pure-slow-full-expansion}, up to a constant unit phase.
\end{thm}

To construct this profile, we first build a finite slow-tail expansion and bound its residual. Next, using the two linear modes at infinity, we construct a right inverse whose corrections have zero limiting modal coefficients, and use it to obtain an exact exterior solution with $C_s=K^{1/(p-1)}$ and $C_f=0$. The exterior construction is then shown to have a full slow expansion and to depend continuously on the parameters.

We then construct a finite Taylor approximation at the origin and use a contraction argument for the integral equation to enclose the exact regular solution starting from the origin and its initial data. We then define the exact and finite three-component matching maps; a zero of the exact map joins the inner and outer solutions up to a constant unit phase. Finally, flow estimates propagate the finite inner and outer enclosures, the computer-assisted verification checks the resulting boundary, propagation, nonvanishing, and matching bounds, and the Brouwer fixed point theorem yields the required matching zero.

\subsection{Exterior solution at infinity}
We now start by constructing a finite approximation to the slow exterior solution for the whole nonlinear equation. The leading slow asymptotic mode of the linear exterior equation $L_\Omega q=0$ is $K^{1/(p-1)}r^\beta$, where $\beta=-2/(p-1)+2i\Omega$. We correct this mode by a finite expansion in inverse even powers of $r$, chosen so that the first $J$ orders of the nonlinear profile equation cancel. For $J\geq1$, define
\begin{equation}\label{eq:pure-slow-logarithmic-approximation}
	P_{\mathrm{out}}^{(J)}(r)=K^{\frac{1}{p-1}}r^\beta e^{\Phi_J(r^{-2})},\quad
	\Phi_J(t)=\sum_{n=1}^J\ell_nt^n.
\end{equation}

To derive a recursive relation for the coefficients, write $t=r^{-2}$ and temporarily regard $\Phi(t)=\sum_{n\geq1}\ell_nt^n$ as a formal power series. The logarithmic derivative of the corresponding formal ansatz with respect to $\log r$ is
\begin{equation*}
	\beta-2t\Phi'(t)=:F(t)=\sum_{n\geq0}f_nt^n,
\end{equation*}
and hence
\begin{equation*}
	f_0=\beta,\qquad f_n=-2n\ell_n,\qquad n\geq1.
\end{equation*}
We now define the two remaining series that enter after substitution:
\begin{align}
	e^{(p-1)\operatorname{Re}\Phi(t)} & =\sum_{n\geq0}E_nt^n,
	\label{eq:pure-slow-E-generating-function}                \\
	F(t)^2+(d-2)F(t)-2tF'(t)          & =\sum_{n\geq0}H_nt^n.
	\label{eq:pure-slow-H-generating-function}
\end{align}
Differentiating \eqref{eq:pure-slow-E-generating-function} and comparing coefficients gives
\begin{equation}\label{eq:pure-slow-E-recursion}
	E_0=1,\qquad
	E_n=\frac{p-1}n\sum_{j=1}^n j\operatorname{Re}\ell_j\,E_{n-j},
	\qquad n\geq1.
\end{equation}
Expanding \eqref{eq:pure-slow-H-generating-function} gives directly
\begin{equation*}
	H_n=\sum_{j=0}^n f_jf_{n-j}+(d-2-2n)f_n,
	\qquad n\geq0,
\end{equation*}
Finally, after division by the corresponding formal ansatz, equation \eqref{eq:complex-profile-radial} formally reads
\begin{equation*}
	t\left(\sum_{n\geq0}(H_n-KE_n)t^n-i\Phi'(t)\right) = 0.
\end{equation*}
Thus the coefficient of $t^{n+1}$ vanishes precisely when
\begin{equation}\label{eq:pure-slow-ell-recursion}
	\ell_{n+1}=\frac{i}{n+1}(KE_n-H_n),
	\qquad n\geq0.
\end{equation}
In particular, the case $n=0$ yields $\ell_1=i\{K-\beta(\beta+d-2)\}$. Thereafter, the coefficients are determined successively: once $\ell_1,\ldots,\ell_n$ are known, so are $E_n$ and $H_n$, and the last formula determines $\ell_{n+1}$. The approximation \eqref{eq:pure-slow-logarithmic-approximation} retains the first $J$ coefficients of this formal expansion. The next part of the construction concerns the analysis of the residual of this finite approximation.

We next estimate the residual of $P_{\mathrm{out}}^{(J)}$ on $[r_{\mathrm{out}},\infty)$. In the definitions above, we now take $\Phi=\Phi_J$ and hence regard $F$ and $H$ as polynomials. We set the rescaled variables
\begin{equation*}
	u=(r_{\mathrm{out}}/r)^2,\qquad L_n=\ell_nr_{\mathrm{out}}^{-2n}\quad (1\leq n\leq J),
\end{equation*}
and introduce the rescaled polynomials
\begin{equation*}
	\Phi_{r_{\mathrm{out}}}(u):=\Phi_J(r_{\mathrm{out}}^{-2}u)=\sum_{n=1}^J L_nu^n,
	\qquad G(u):=H(r_{\mathrm{out}}^{-2}u)=\sum_{n=0}^{2J}G_nu^n.
\end{equation*}
Thus the new variable satisfies $0\leq u\leq1$ and $\Phi_J(r^{-2})=\Phi_{r_{\mathrm{out}}}(u)$. For $\nu\in\mathbb R$, let $X_{\nu,r_{\mathrm{out}}}$ be the space of continuous complex functions on $[r_{\mathrm{out}},\infty)$ with weighted norm
\begin{equation*}
	\|f\|_{\nu,r_{\mathrm{out}}}:=\sup_{r\geq r_{\mathrm{out}}}r^\nu|f(r)|.
\end{equation*}
The preceding coefficient recursions yield the following estimates. We need to have explicit constants in the bounds to use them in the computer-assisted proof.

\begin{lem}
	\label{lem:pure-slow-finite-residual}
	Recall that $N(q)=|q|^{p-1}q$. For any $\chi>1$, define
	\begin{equation}\label{eq:pure-slow-residual-majorant}
		\mathcal B
		=\sum_{n=J}^{2J}|G_n|
		+K\chi^{-J}
		\exp\left((p-1)\sum_{n=1}^J
		|\operatorname{Re}L_n|\chi^n\right),  \qquad
		M=K^{\frac{1}{p-1}}\exp\left(\sum_{n=1}^J|\operatorname{Re}L_n|\right),
	\end{equation}
	Then the residual of $P_{\mathrm{out}}^{(J)}$ is
	\begin{equation}\label{eq:pure-slow-residual-identity}
		L_\Omega P_{\mathrm{out}}^{(J)}-N(P_{\mathrm{out}}^{(J)})
		=P_{\mathrm{out}}^{(J)}(r)\left\{
		\frac{u}{r_{\mathrm{out}}^2}
		\left(G(u)-Ke^{(p-1)\operatorname{Re}\Phi_{r_{\mathrm{out}}}(u)}\right)
		-iu\Phi_{r_{\mathrm{out}}}'(u)\right\}.
	\end{equation}
	Moreover, for $0\leq u\leq1$, we have the estimates
	\begin{equation}\label{eq:pure-slow-residual-bound}
		\begin{aligned}
			\left|\frac{u}{r_{\mathrm{out}}^2}
			\left(G(u)-Ke^{(p-1)\operatorname{Re}\Phi_{r_{\mathrm{out}}}(u)}\right)
			-iu\Phi_{r_{\mathrm{out}}}'(u)\right|
			 & \leq r_{\mathrm{out}}^{-2}\mathcal B\,u^{J+1}, \\
			\|L_\Omega P_{\mathrm{out}}^{(J)}-N(P_{\mathrm{out}}^{(J)})\|_{\nu,r_{\mathrm{out}}}
			 & \leq Mr_{\mathrm{out}}^{2J}\mathcal B,
			\qquad \nu=\frac{2}{p-1}+2J+2.
		\end{aligned}
	\end{equation}
\end{lem}

\begin{proof}
	Direct substitution of $P_{\mathrm{out}}^{(J)}$ into the profile equation, using the definitions of $F$ and $G$, gives \eqref{eq:pure-slow-residual-identity}. The coefficient recursions from above make every coefficient through degree $J$ of the expression in braces vanish.

	Write
	\begin{equation*}
		e^{(p-1)\operatorname{Re}\Phi_{r_{\mathrm{out}}}(u)}=\sum_{m=0}^{\infty}e_mu^m,
		\qquad
		\exp\left((p-1)\sum_{n=1}^J
		|\operatorname{Re}L_n|u^n\right)=\sum_{m=0}^{\infty}b_mu^m.
	\end{equation*}
	The Cauchy product formula gives $|e_m|\leq b_m$ for every $m\geq0$, and $b_m\geq0$. After factoring out $u/r_{\mathrm{out}}^2$, the term involving $\Phi_{r_{\mathrm{out}}}'$ has degree at most $J-1$. Hence the cancellations through degree $J$ in the expression in braces are equivalent to cancellation of the coefficients of degrees $0,\ldots,J-1$ in the bracket, and thus
	\begin{equation*}
		G(u)-Ke^{(p-1)\operatorname{Re}\Phi_{r_{\mathrm{out}}}(u)}-ir_{\mathrm{out}}^2\Phi_{r_{\mathrm{out}}}'(u)
		=\sum_{m=J}^{2J}G_mu^m-K\sum_{m=J}^{\infty}e_mu^m.
	\end{equation*}
	For $0\leq u\leq1$ and $\chi>1$, it follows that
	\begin{align*}
		\left|G(u)-Ke^{(p-1)\operatorname{Re}\Phi_{r_{\mathrm{out}}}(u)}
		-ir_{\mathrm{out}}^2\Phi_{r_{\mathrm{out}}}'(u)\right|
		 & \leq u^J\left(\sum_{m=J}^{2J}|G_m|
		+K\sum_{m=J}^{\infty}|e_m|\right)              \\
		 & \leq u^J\left(\sum_{m=J}^{2J}|G_m|
		+K\chi^{-J}\sum_{m=J}^{\infty}b_m\chi^m\right) \\
		 & \leq \mathcal B u^J,
	\end{align*}
	because $\sum_{m=0}^{\infty}b_m\chi^m =\exp((p-1)\sum_{n=1}^J|\operatorname{Re}L_n|\chi^n)$. Multiplication by $u/r_{\mathrm{out}}^2$ proves the first bound in \eqref{eq:pure-slow-residual-bound}.

	Finally, for $0\leq u\leq1$, $|\operatorname{Re}\Phi_{r_{\mathrm{out}}}(u)|\leq \sum_{n=1}^J|\operatorname{Re}L_n|$. Hence $|P_{\mathrm{out}}^{(J)}(r)|\leq Mr^{-2/(p-1)}$ for $r\geq r_{\mathrm{out}}$. Combining this amplitude bound with $u=(r_{\mathrm{out}}/r)^2$ gives the second estimate in \eqref{eq:pure-slow-residual-bound}.
\end{proof}

To correct $P_{\mathrm{out}}^{(J)}$ to an exact exterior solution, write $Q^{\mathrm{out}}=P_{\mathrm{out}}^{(J)}+h$. Then
\begin{equation*}
	L_\Omega h=N(P_{\mathrm{out}}^{(J)}+h)-N(P_{\mathrm{out}}^{(J)})-\mathcal R_{\mathrm{out}},
\end{equation*}
where the exterior residual is $\mathcal R_{\mathrm{out}}(r) :=(L_\Omega P_{\mathrm{out}}^{(J)})(r) -N(P_{\mathrm{out}}^{(J)}(r))$. Here $L_\Omega$ is the linear exterior operator from \eqref{eq:profile-linear-operator}. We therefore need an inverse for $L_\Omega$ whose output has zero limiting slow and fast modal coefficients, so that the prescribed values of $C_s$ and $C_f$ are unchanged.

Recall the weighted space $X_{\nu,r_{\mathrm{out}}}$ defined above. For $f\in X_{\nu,r_{\mathrm{out}}}$, use the variation of parameters formula to define
\begin{equation*}
	(\mathcal L_{\Omega,\infty}^{-1}f)(r)
	:={}
	S_\Omega(r)
	\int_r^\infty
	\frac{F_\Omega(\sigma)f(\sigma)}{W_\Omega(\sigma)}\,d\sigma\\
	-
	F_\Omega(r)
	\int_r^\infty
	\frac{S_\Omega(\sigma)f(\sigma)}{W_\Omega(\sigma)}\,d\sigma.
\end{equation*}
The following lemma shows that this gives a right inverse to $L_\Omega$ on the weighted space $X_{\nu,r_{\mathrm{out}}}$ and provides explicit bounds for the inverse. These bounds must be compatible with the computer-assisted verification, so the constants are tracked explicitly.

\begin{lem}\label{lem:pure-slow-inverse}
	Suppose $r_{\mathrm{out}}\geq32$,
	$\nu>d-\frac{2}{p-1}$, and $|\Omega|\leq\Omega_{\mathrm{bd}}$, where
	\begin{equation*}
		\Omega_{\mathrm{bd}}
		:=
		\begin{cases}
			7/5,   & (d,p)=(3,27), \\
			59/20, & (d,p)=(4,9),  \\
			7/4,   & (d,p)=(5,7).
		\end{cases}
	\end{equation*}
	Then, for every $f\in X_{\nu,r_{\mathrm{out}}}$, the function $\mathcal L_{\Omega,\infty}^{-1}f$ belongs to $C^2([r_{\mathrm{out}},\infty))$ and satisfies the equality $L_\Omega\mathcal L_{\Omega,\infty}^{-1}f=f$. Moreover, there is an explicit constant $\mathfrak C_{\nu,r_{\mathrm{out}}}$, independent of $\Omega$ and $f$, such that
	\begin{equation}\label{eq:pure-slow-inverse-bounds}
		\|\mathcal L_{\Omega,\infty}^{-1}f
		\|_{\nu,r_{\mathrm{out}}}
		\leq
		\mathfrak C_{\nu,r_{\mathrm{out}}}
		\|f\|_{\nu,r_{\mathrm{out}}},\qquad
		\|(\mathcal L_{\Omega,\infty}^{-1}f)'
		\|_{\nu-1,r_{\mathrm{out}}}
		\leq
		\mathfrak C_{\nu,r_{\mathrm{out}}}
		\|f\|_{\nu,r_{\mathrm{out}}}.
	\end{equation}
\end{lem}

\begin{proof}
	Recall $\mu$ and $\kappa$ defined by \eqref{eq:mu-and-kappa}. We first normalize $S_\Omega$ and $F_\Omega$ using these parameters and derive estimates on the rescaled functions. Let
	\begin{equation*}
		\widehat S_\Omega(r):=r^{2\mu}S_\Omega(r),\qquad
		\widehat F_\Omega(r):=e^{\frac{ir^2}{4}}r^{2\kappa}F_\Omega(r).
	\end{equation*}
	Rescaling \eqref{eq:profile-general-modes}, we obtain $\widehat S_\Omega=H_\mu^+$ and $\widehat F_\Omega=H_\kappa^-$, where
	\begin{equation*}
		H_z^\pm(r):=\frac{1}{\Gamma(z)}\int_0^\infty
		e^{-t}t^{z-1}\left(1\pm\frac{4it}{r^2}\right)^{\frac d2-1-z}\,dt.
	\end{equation*}
	We estimate both integrals at once. Write $z=x+i\omega$, with $x>0$ and $|\omega|\leq\Omega_{\mathrm{bd}}$, and set
	\begin{equation*}
		\alpha:=\frac d2-1-z,\qquad
		c(x):=\Omega_{\mathrm{bd}}+\max\left\{0,\frac d2-2-x\right\}.
	\end{equation*}
	Since $|\arg(1\pm iv)|\leq v$ and $\log|1\pm iv|\leq v$ for $v\geq0$, the principal powers satisfy
	\begin{equation*}
		|(1\pm iv)^{\alpha-1}|
		=\exp\left[\left(\frac d2-2-x\right)\log|1\pm iv|
			+\omega\arg(1\pm iv)\right]
		\leq e^{c(x)v}.
	\end{equation*}
	Integrating the derivative in $v$ therefore gives
	\begin{equation}\label{eq:principal-power-increment-bound}
		|(1\pm iv)^\alpha-1|
		\leq |\alpha|\int_0^v e^{c(x)w}\,dw
		\leq |\alpha|v e^{c(x)v}.
	\end{equation}
	For the stated parameter pairs, $c(x)\leq59/20+1/2<32^2/4$. Thus $\delta_x:=1-4c(x)/32^2>0$. Subtracting the Gamma integral $\Gamma(z)^{-1}\int_0^\infty e^{-t}t^{z-1}\,dt=1$, we infer, for $r\geq32$,
	\begin{equation*}
		\begin{aligned}
			|H_z^\pm(r)-1|
			 & \leq\frac{4|\alpha|}{r^2|\Gamma(z)|}
			\int_0^\infty e^{-(1-4c(x)/r^2)t}t^x\,dt        \\
			 & =\frac{4|\alpha|\Gamma(x+1)}{r^2|\Gamma(z)|}
			\left(1-\frac{4c(x)}{r^2}\right)^{-x-1}
			\leq\frac{4|\alpha|\Gamma(x+1)}{r^2|\Gamma(z)|}
			\delta_x^{-x-1}.
		\end{aligned}
	\end{equation*}
	To estimate the derivative, we compute
	\begin{equation*}
		r\frac{d}{dr}\left(1\pm\frac{4it}{r^2}\right)^\alpha
		=\mp\frac{8i\alpha t}{r^2}
		\left(1\pm\frac{4it}{r^2}\right)^{\alpha-1}.
	\end{equation*}
	The derivative integrand is bounded by a constant times $t^x e^{-\delta_x t}$, uniformly for $r\geq32$; the original integrand is bounded by a constant times $e^{-t}t^{x-1}+t^x e^{-\delta_x t}$, by \eqref{eq:principal-power-increment-bound}. These integrable bounds justify differentiation under the integral and yield
	\begin{equation*}
		|r(H_z^\pm)'(r)|
		\leq\frac{8|\alpha|}{r^2|\Gamma(z)|}
		\int_0^\infty t^x e^{-\delta_x t}\,dt\\
		=\frac{8|\alpha|\Gamma(x+1)}{r^2|\Gamma(z)|}
		\delta_x^{-x-1}.
	\end{equation*}

	It remains to bound the Gamma quotient uniformly in $\omega$. Euler's product formula gives
	\begin{equation*}
		\frac{\Gamma(x)}{|\Gamma(x+i\omega)|}
		=\prod_{n=0}^\infty
		\left(1+\frac{\omega^2}{(x+n)^2}\right)^{1/2}
		\leq R(x),
	\end{equation*}
	where
	\begin{equation*}
		R(x):=\prod_{n=0}^{15}
		\left(1+\frac{\Omega_{\mathrm{bd}}^2}{(x+n)^2}\right)^{1/2}
		\exp\left[\frac{\Omega_{\mathrm{bd}}^2}{2}
			\left((x+16)^{-2}+(x+16)^{-1}\right)\right].
	\end{equation*}
	Indeed, $\log(1+t)\leq t$ bounds the logarithm of the product tail, and the integral comparison gives
	\begin{equation*}
		\sum_{n=16}^\infty(x+n)^{-2}
		\leq(x+16)^{-2}+\int_{16}^\infty(x+t)^{-2}\,dt
		=(x+16)^{-2}+(x+16)^{-1}.
	\end{equation*}
	Define the explicit constants
	\begin{equation*}
		\begin{aligned}
			D(x) & :=4x\sqrt{\left(\frac d2-1-x\right)^2
			+\Omega_{\mathrm{bd}}^2}\,R(x)\delta_x^{-x-1}, \\
			D_S  & :=D\left(\frac{1}{p-1}\right),\qquad
			D_F:=D\left(\frac d2-\frac{1}{p-1}\right).
		\end{aligned}
	\end{equation*}
	Using $\Gamma(x+1)=x\Gamma(x)$ and applying the preceding bounds with $z=\mu,\kappa$, we obtain
	\begin{equation}\label{eq:pure-slow-linear-mode-bounds}
		\begin{aligned}
			|\widehat S_\Omega(r)-1| & \leq D_Sr^{-2},
			                         & |r\widehat S_\Omega'(r)| & \leq2D_Sr^{-2}, \\
			|\widehat F_\Omega(r)-1| & \leq D_Fr^{-2},
			                         & |r\widehat F_\Omega'(r)| & \leq2D_Fr^{-2}.
		\end{aligned}
	\end{equation}

	Returning to the original modes, we compute
	\begin{equation*}
		S_\Omega'(r)=r^{-2\mu-1}
		\left(-2\mu\widehat S_\Omega+r\widehat S_\Omega'\right), \qquad
		F_\Omega'(r)=e^{-\frac{ir^2}{4}}r^{-2\kappa}
		\left[\widehat F_\Omega'
			-\left(\frac{ir}{2}+\frac{2\kappa}{r}\right)\widehat F_\Omega\right].
	\end{equation*}
	Hence \eqref{eq:pure-slow-linear-mode-bounds}, together with $|\mu|\leq\sqrt{\frac{1}{(p-1)^2}+\Omega_{\mathrm{bd}}^2}$ and $|\kappa|\leq\sqrt{\left(\frac d2-\frac{1}{p-1}\right)^2+\Omega_{\mathrm{bd}}^2}$, yields
	\begin{equation}\label{eq:pure-slow-mode-bounds}
		\begin{aligned}
			r^{\frac{2}{p-1}}|S_\Omega(r)|   & \leq A_{r_{\mathrm{out}}},
			                                 & r^{\frac{2}{p-1}+1}|S_\Omega'(r)|   & \leq A_{r_{\mathrm{out}}}, \\
			r^{d-\frac{2}{p-1}}|F_\Omega(r)| & \leq A_{r_{\mathrm{out}}},
			                                 & r^{d-\frac{2}{p-1}-1}|F_\Omega'(r)| & \leq A_{r_{\mathrm{out}}},
		\end{aligned}
	\end{equation}
	for $r\geq r_{\mathrm{out}}$, where we choose
	\begin{equation*}
		\begin{aligned}
			A_{r_{\mathrm{out}}}:=\max\Biggl\{ &
			1+\frac{D_S}{r_{\mathrm{out}}^2},\quad
			1+\frac{D_F}{r_{\mathrm{out}}^2},                                                     \\
			                                   & 2\sqrt{\frac{1}{(p-1)^2}+\Omega_{\mathrm{bd}}^2}
			\left(1+\frac{D_S}{r_{\mathrm{out}}^2}\right)
			+\frac{2D_S}{r_{\mathrm{out}}^2},                                                     \\
			                                   & \left(\frac12+
			\frac{2\sqrt{\left(\frac d2-\frac{1}{p-1}\right)^2+\Omega_{\mathrm{bd}}^2}}{r_{\mathrm{out}}^2}\right)
			\left(1+\frac{D_F}{r_{\mathrm{out}}^2}\right)
			+\frac{2D_F}{r_{\mathrm{out}}^4}\Biggr\}.
		\end{aligned}
	\end{equation*}
	Here we used $r^{-j}\leq r_{\mathrm{out}}^{-j}$ for $j=2,4$ to make the bounds uniform on the exterior interval.

	Now we will show that the inverse is well-defined and use the bounds on the unscaled modes to derive bounds on this inverse. Let $f\in X_{\nu,r_{\mathrm{out}}}$ and write
	\begin{equation*}
		I_F(r):=\int_r^\infty\frac{F_\Omega(\sigma)f(\sigma)}{W_\Omega(\sigma)}\,d\sigma,
		\qquad
		I_S(r):=\int_r^\infty\frac{S_\Omega(\sigma)f(\sigma)}{W_\Omega(\sigma)}\,d\sigma.
	\end{equation*}
	For each of the three parameter pairs, $d-\frac{2}{p-1}>\frac{2}{p-1}$, so the assumption on $\nu$ implies both $\nu>d-\frac{2}{p-1}$ and $\nu>\frac{2}{p-1}$. Using $|W_\Omega(r)|=r^{1-d}/2$ from \eqref{eq:profile-Wronskian} and \eqref{eq:pure-slow-mode-bounds}, we obtain the absolutely convergent tail estimates
	\begin{equation}\label{eq:pure-slow-tail-bounds}
		\begin{aligned}
			|I_F(r)|
			 & \leq2A_{r_{\mathrm{out}}}\|f\|_{\nu,r_{\mathrm{out}}}
			\int_r^\infty \sigma^{\frac{2}{p-1}-\nu-1}\,d\sigma
			=\frac{2A_{r_{\mathrm{out}}}}{\nu-\frac{2}{p-1}}
			\|f\|_{\nu,r_{\mathrm{out}}}r^{\frac{2}{p-1}-\nu},       \\
			|I_S(r)|
			 & \leq2A_{r_{\mathrm{out}}}\|f\|_{\nu,r_{\mathrm{out}}}
			\int_r^\infty \sigma^{d-\frac{2}{p-1}-\nu-1}\,d\sigma
			=\frac{2A_{r_{\mathrm{out}}}}{\nu+\frac{2}{p-1}-d}
			\|f\|_{\nu,r_{\mathrm{out}}}r^{d-\frac{2}{p-1}-\nu}.
		\end{aligned}
	\end{equation}
	Thus $u:=\mathcal L_{\Omega,\infty}^{-1}f=S_\Omega I_F-F_\Omega I_S$ is well-defined. Since $f$ is continuous, variation of parameters gives $u\in C^2([r_{\mathrm{out}},\infty))$ and $L_\Omega u=f$. In particular, the endpoint terms cancel in the first derivative:
	\begin{equation*}
		u'=S_\Omega'I_F-F_\Omega'I_S.
	\end{equation*}
	Combining \eqref{eq:pure-slow-mode-bounds} and \eqref{eq:pure-slow-tail-bounds} with these formulas yields
	\begin{equation*}
		\begin{aligned}
			r^\nu|u(r)|
			 & \leq2A_{r_{\mathrm{out}}}^2
			\left(\frac{1}{\nu-\frac{2}{p-1}}+\frac{1}{\nu+\frac{2}{p-1}-d}\right)
			\|f\|_{\nu,r_{\mathrm{out}}},  \\
			r^{\nu-1}|u'(r)|
			 & \leq2A_{r_{\mathrm{out}}}^2
			\left(\frac{r^{-2}}{\nu-\frac{2}{p-1}}+\frac{1}{\nu+\frac{2}{p-1}-d}\right)
			\|f\|_{\nu,r_{\mathrm{out}}}.
		\end{aligned}
	\end{equation*}
	Since $r\geq32$, taking suprema proves \eqref{eq:pure-slow-inverse-bounds} with
	\begin{equation*}
		\mathfrak C_{\nu,r_{\mathrm{out}}}
		:=2A_{r_{\mathrm{out}}}^2
		\left(\frac{1}{\nu-\frac{2}{p-1}}
		+\frac{1}{\nu+\frac{2}{p-1}-d}\right).
	\end{equation*}
	All constants depend only on $d$, $p$, $\Omega_{\mathrm{bd}}$, $r_{\mathrm{out}}$, and $\nu$, and are therefore independent of $\Omega$ and $f$.
\end{proof}

We now use this right inverse to correct the finite approximation in the weighted exterior space $X_{\nu,r_{\mathrm{out}}}$. More precisely, we seek an additive correction $h$ such that $Q^{\mathrm{out}}=P_{\mathrm{out}}^{(J)}+h$ solves the full nonlinear profile equation. We solve for this correction using a fixed point argument. Its weighted decay will preserve $C_s=K^{1/(p-1)}$ and $C_f=0$.

\begin{prop}\label{prop:pure-slow-exterior-certificate}
	Suppose $r_{\mathrm{out}}\geq32$, $K>0$, $J\geq2$, and $\Omega$ satisfies the
	frequency bound in Lemma~\ref{lem:pure-slow-inverse}. Set
	$\nu=2/(p-1)+2J+2>d-2/(p-1)$.
	If there is a $\delta_{\mathrm{ext}}>0$ such that
	\begin{equation}\label{eq:pure-slow-exterior-tests}
		\begin{aligned}
			 & p\mathfrak C_{\nu,r_{\mathrm{out}}}r_{\mathrm{out}}^{-2}(M+\delta_{\mathrm{ext}})^{p-1}<1, \\
			 & \frac{\mathfrak C_{\nu,r_{\mathrm{out}}}M\mathcal B}{r_{\mathrm{out}}^2}
			\leq\left[1-p\mathfrak C_{\nu,r_{\mathrm{out}}}r_{\mathrm{out}}^{-2}
			(M+\delta_{\mathrm{ext}})^{p-1}\right]\delta_{\mathrm{ext}},                                  \\
			 & \delta_{\mathrm{ext}}<\frac{K^{\frac{2}{p-1}}}{M},
		\end{aligned}
	\end{equation}
	then there is a smooth nonvanishing exterior solution $Q^{\mathrm{out}}$ with $C_s=K^{1/(p-1)}$ and $C_f=0$. For $r\geq r_{\mathrm{out}}$ it satisfies
	\begin{equation}\label{eq:pure-slow-exterior-enclosure}
		|Q^{\mathrm{out}}-P_{\mathrm{out}}^{(J)}|
		\leq \delta_{\mathrm{ext}}r_{\mathrm{out}}^{2J+2}r^{-\frac{2}{p-1}-2J-2},\qquad
		|(Q^{\mathrm{out}})'-(P_{\mathrm{out}}^{(J)})'|
		\leq \delta_{\mathrm{ext}}r_{\mathrm{out}}^{2J+2}r^{-\frac{2}{p-1}-2J-1}.
	\end{equation}
\end{prop}
\begin{proof}
	We apply the inverse from Lemma~\ref{lem:pure-slow-inverse} to the correction equation. The three conditions in \eqref{eq:pure-slow-exterior-tests} ensure, respectively, contraction, invariance of the correction ball, and nonvanishing of the resulting profile.

	Since $\operatorname{Re}\beta=-2/(p-1)$,
	\begin{equation*}
		|P_{\mathrm{out}}^{(J)}(r)|=K^{\frac{1}{p-1}}r^{-\frac{2}{p-1}}e^{\operatorname{Re}\Phi_{r_{\mathrm{out}}}(u)}.
	\end{equation*}
	For $0\leq u\leq1$,
	\begin{equation*}
		|\operatorname{Re}\Phi_{r_{\mathrm{out}}}(u)|
		\leq\sum_{n=1}^J|\operatorname{Re}L_n|u^n
		\leq\sum_{n=1}^J|\operatorname{Re}L_n|.
	\end{equation*}
	Therefore
	\begin{equation*}
		|P_{\mathrm{out}}^{(J)}(r)|\geq \frac{K^{\frac{2}{p-1}}}{M}r^{-\frac{2}{p-1}}.
	\end{equation*}
	This lower bound will prevent the correction from producing a zero.

	Consider the closed ball $\{h\in X_{\nu,r_{\mathrm{out}}}:\|h\|_{\nu,r_{\mathrm{out}}} \leq\delta_{\mathrm{ext}}r_{\mathrm{out}}^{2J+2}\}$. For $h$ in this ball and $r\geq r_{\mathrm{out}}$,
	\begin{equation*}
		|h(r)|\leq\delta_{\mathrm{ext}}r_{\mathrm{out}}^{2J+2}r^{-\nu}
		=\delta_{\mathrm{ext}}r_{\mathrm{out}}^{2J+2}r^{-\frac{2}{p-1}-2J-2}
		\leq\delta_{\mathrm{ext}}r^{-\frac{2}{p-1}}.
	\end{equation*}
	Consequently, $|P_{\mathrm{out}}^{(J)}+h|\leq (M+\delta_{\mathrm{ext}})r^{-2/(p-1)}$. We consider the map
	\begin{equation}\label{eq:pure-slow-exterior-fixed-point}
		h\longmapsto\mathcal L_{\Omega,\infty}^{-1}
		\{N(P_{\mathrm{out}}^{(J)}+h)-N(P_{\mathrm{out}}^{(J)})-\mathcal R_{\mathrm{out}}\}.
	\end{equation}
	Since $\|DN(q)\|\leq p|q|^{p-1}$, the first condition in \eqref{eq:pure-slow-exterior-tests} makes \eqref{eq:pure-slow-exterior-fixed-point} a contraction. At zero, the inverse estimate and \eqref{eq:pure-slow-residual-bound} give the norm bound $\mathfrak C_{\nu,r_{\mathrm{out}}}Mr_{\mathrm{out}}^{2J}\mathcal B$. Thus the second condition in \eqref{eq:pure-slow-exterior-tests} makes this ball invariant. Banach's fixed-point theorem gives $h$, and the derivative estimate in Lemma~\ref{lem:pure-slow-inverse} yields \eqref{eq:pure-slow-exterior-enclosure}.

	The inverse identity gives the profile equation. Since
	\begin{equation*}
		|h(r)|\leq\delta_{\mathrm{ext}}r^{-\frac{2}{p-1}}
		<K^{\frac{2}{p-1}}M^{-1}r^{-\frac{2}{p-1}}
		\leq|P_{\mathrm{out}}^{(J)}(r)|,
	\end{equation*}
	the profile is nonvanishing. Smoothness follows from the equation. Also, $r^{2\mu}Q^{\mathrm{out}}\to K^{1/(p-1)}$ by the enclosure. Moreover, for $J\geq2$ the approximation $P_{\mathrm{out}}^{(J)}$ agrees with $|C_s|r^\beta(1+c_1r^{-2}+c_2r^{-4})$, the slow part in \eqref{eq:fast-coefficient-definition}, up to $O(r^{-2/(p-1)-6})$, where $c_1=\ell_1$ and $c_2=\ell_2+\ell_1^2/2$. Both this difference and $h$ are $o(r^{-d+2/(p-1)})$. The definition of the nonlinear fast coefficient therefore gives $C_f=0$.
\end{proof}

In this proposition the parameter $\delta_{\mathrm{ext}}$ measures the size of the correction. In the following argument, we will use two different choices of this parameter. The first one will be used to show uniqueness of the exterior solution, and the second choice will be used inside the computer-assisted argument.

The exterior construction above is initially carried out at one finite order. To use the resulting profile in the later global estimates, we need a single exterior solution with an expansion to every algebraic order, rather than a different solution at each truncation order.

\begin{cor}
	\label{cor:pure-slow-full-expansion}
	The exterior solution constructed in
	Proposition~\ref{prop:pure-slow-exterior-certificate} satisfies,
	for every integer $J\geq1$,
	\begin{equation}\label{eq:pure-slow-full-expansion}
		\begin{aligned}
			Q^{\mathrm{out}}    & =P_{\mathrm{out}}^{(J)}+O(r^{-\frac{2}{p-1}-2J-2}), \\
			(Q^{\mathrm{out}})' & =(P_{\mathrm{out}}^{(J)})'
			+O(r^{-\frac{2}{p-1}-2J-3}).
		\end{aligned}
	\end{equation}
	In particular, its algebraic asymptotic expansion contains only the slow mode.
\end{cor}
\begin{proof}
	Fix an order $J\geq2$. With $\chi=2$, the residual constant in \eqref{eq:pure-slow-residual-majorant} remains bounded as $r_{\mathrm{out}}\to\infty$, while $M,K^{2/(p-1)}/M\to K^{1/(p-1)}$. Set $\delta_{\mathrm{ext}}=2\mathfrak C_{\nu,r_{\mathrm{out}}}M\mathcal B/r_{\mathrm{out}}^2$. For sufficiently large $r_{\mathrm{out}}$, both $\delta_{\mathrm{ext}}$ and the contraction constant tend to zero. Hence the three conditions in \eqref{eq:pure-slow-exterior-tests} hold. Thus, for each fixed $J$, Proposition~\ref{prop:pure-slow-exterior-certificate} applies and gives an exact exterior solution
	\begin{equation*}
		Q_{\mathrm{out}}^{(J)}=P_{\mathrm{out}}^{(J)}
		+O(r^{-\frac{2}{p-1}-2J-2}),\qquad r\geq R_J.
	\end{equation*}
	We now show that these constructions, that depend on $J$, define one solution. Let $Q_{\mathrm{out}}^{(J_1)}$ and $Q_{\mathrm{out}}^{(J_2)}$ be two of them, and work on the common region $r\geq R_*:=\max\{R_{J_1},R_{J_2}\}$. Since $J_1, J_2 \geq 2$, their difference $w$ satisfies
	\begin{equation*}
		L_\Omega w=N(Q_{\mathrm{out}}^{(J_1)})
		-N(Q_{\mathrm{out}}^{(J_2)})=:f,\qquad
		w=O(r^{-\frac{2}{p-1}-6}),\qquad f=O(r^{-\frac{2}{p-1}-8}).
	\end{equation*}
	Since both terms in $w-\mathcal L_{\Omega,\infty}^{-1}f$ decay faster than the slow and fast modes, its coefficients in the fundamental system $\{S_\Omega,F_\Omega\}$ vanish. Then this homogeneous solution is zero, and hence $w=\mathcal L_{\Omega,\infty}^{-1}f$. If $|Q_{\mathrm{out}}^{(J_1)}|,|Q_{\mathrm{out}}^{(J_2)}| \leq M_*r^{-2/(p-1)}$, then
	\begin{equation*}
		|f|\leq pM_*^{p-1}r^{-2}|w|.
	\end{equation*}
	Lemma~\ref{lem:pure-slow-inverse} consequently gives
	\begin{equation*}
		\|w\|_{2/(p-1)+6,r_{\mathrm{out}}}
		\leq p\mathfrak C_{2/(p-1)+6,r_{\mathrm{out}}}M_*^{p-1}r_{\mathrm{out}}^{-2}
		\|w\|_{2/(p-1)+6,r_{\mathrm{out}}}.
	\end{equation*}
	For sufficiently large $R_*$ the coefficient is less than one, so $w=0$ on the common exterior region. ODE uniqueness then extends this equality throughout the common interval of regularity. We denote the resulting common solution simply by $Q^{\mathrm{out}}$. The solution certified at the computational $r_{\mathrm{out}}$ satisfies the same tail bounds by \eqref{eq:pure-slow-exterior-enclosure}, so this uniqueness argument identifies it with $Q^{\mathrm{out}}$, first on a tail and then on $[r_{\mathrm{out}},\infty)$ by ODE uniqueness.

	The estimate for $Q_{\mathrm{out}}^{(J)}$ now applies to this common $Q^{\mathrm{out}}$, proving the first line of \eqref{eq:pure-slow-full-expansion}. Applying the derivative estimate in \eqref{eq:pure-slow-exterior-enclosure} at order $J+1$, and using $(P_{\mathrm{out}}^{(J+1)})'-(P_{\mathrm{out}}^{(J)})' =O(r^{-2/(p-1)-2J-3})$, proves the second line. The case $J=1$ follows by truncating the expansion at order two, for both the value and the derivative.
\end{proof}

The matching argument that will be used varies $(\Omega,K)$ over the parameter cube $D$. It therefore requires $\bigl(Q^{\mathrm{out}}(r_{\mathrm{out}}), (Q^{\mathrm{out}})'(r_{\mathrm{out}})\bigr)$ to depend continuously on $(\Omega,K)$. The three inequalities in the following corollary are the same as those in \eqref{eq:pure-slow-exterior-tests}, except that they are now required to hold uniformly.

\begin{cor}\label{cor:pure-slow-exterior-family}
	Suppose there is a single $\delta_{\mathrm{ext}}>0$
	such that all three conditions \eqref{eq:pure-slow-exterior-tests}
	hold with this same $\delta_{\mathrm{ext}}$ for every $(\Omega,K)$ satisfying
	$|\Omega-\Omega_0|\leq\rho$ and $|K-K_0|\leq\rho$,
	with $\rho$ as in \eqref{eq:pure-slow-parameter-box}.

	Then the map
	\begin{equation*}
		(\Omega,K)\longmapsto
		\bigl(Q^{\mathrm{out}}(r_{\mathrm{out}}),(Q^{\mathrm{out}})'(r_{\mathrm{out}})\bigr),
	\end{equation*}
	is continuous.
\end{cor}
\begin{proof}
	Let $(\Omega_n,K_n)\to(\Omega,K)$ in the parameter rectangle above, and denote the corresponding exterior solutions by $Q_n^{\mathrm{out}}$. The coefficients of $P_{\mathrm{out}}^{(J)}$ depend continuously on $(\Omega,K)$. Thus \eqref{eq:pure-slow-exterior-enclosure} bounds $Q_n^{\mathrm{out}}$ and $(Q_n^{\mathrm{out}})'$ uniformly on every compact interval of $[r_{\mathrm{out}},\infty)$, and the profile equation also bounds $(Q_n^{\mathrm{out}})''$. By Arzel\`a--Ascoli and a diagonal extraction, every subsequence has a further subsequence converging in $C^1_{\mathrm{loc}}([r_{\mathrm{out}},\infty))$.

	Any such limit solves the profile equation at $(\Omega,K)$ and inherits the enclosure \eqref{eq:pure-slow-exterior-enclosure}, by passing to the limit at each $r\geq r_{\mathrm{out}}$. Its difference from $Q^{\mathrm{out}}$ at $(\Omega,K)$ is therefore $O(r^{-2/(p-1)-2J-2})$, hence $O(r^{-2/(p-1)-6})$ because $J\geq2$. The uniqueness argument in Corollary~\ref{cor:pure-slow-full-expansion} identifies the two solutions on a sufficiently distant tail, and ODE uniqueness extends their equality to $[r_{\mathrm{out}},\infty)$.

	Consequently the entire sequence converges to $Q^{\mathrm{out}}$ in $C^1_{\mathrm{loc}}([r_{\mathrm{out}},\infty))$. Evaluation at $r_{\mathrm{out}}$ proves the claimed continuity.
\end{proof}

\subsection{Interior solution at the origin}
We next construct and enclose the profile starting from its data at the origin. First, an even Taylor polynomial is chosen so that the low orders of the profile equation vanish, leaving an explicitly computable finite residual. We then use the integral formula with zero initial data to correct this polynomial to the exact solution $Q^{\mathrm{in}}$. The resulting value and derivative bounds at $r=r_{\mathrm{in}}$ provide the initial data needed to propagate the inner solution to the matching radius.

Fix $r_{\mathrm{in}}>0$ and an integer $I\geq1$. Define the inner Taylor polynomial, truncated at order $I$, by
\begin{equation*}
	P_{\mathrm{in}}^{(I)}(r)=\sum_{n=0}^{I}c_nr^{2n},\qquad c_0=A,
\end{equation*}
where the coefficients are determined by the recurrence
\begin{equation}\label{eq:pure-slow-origin-recursion}
	\begin{aligned}
		2(n+1)(2n+d)c_{n+1}
		       & =[z^n]\bigl(C_n(z)^{\frac{p+1}{2}}\overline C_n(z)^{\,\frac{p-1}{2}}\bigr)
		-\left(\Omega+i\left(\frac{1}{p-1}+n\right)\right)c_n,                              \\
		C_n(z) & =\sum_{j=0}^n c_jz^j.
	\end{aligned}
\end{equation}
The recurrence is used for $0\leq n<I$. Here $[z^n]$ denotes coefficient extraction, and the bar conjugates coefficients. The residual $\mathcal R_{\mathrm{in}} =L_\Omega P_{\mathrm{in}}^{(I)}-N(P_{\mathrm{in}}^{(I)})$ is a polynomial in $r^2$. Its coefficients below degree $I$ vanish identically. We denote its coefficients by $d_n$, so that
\begin{equation}\label{eq:pure-slow-origin-residual-bound}
	\mathcal R_{\mathrm{in}}(r)=\sum_{n=I}^{pI}d_nr^{2n} \quad \text{and} \quad
	\|\mathcal R_{\mathrm{in}}\|_{L^\infty(0,r_{\mathrm{in}})}
	\leq H_0:=\sum_{n=I}^{pI}|d_n|r_{\mathrm{in}}^{2n}.
\end{equation}

A direct contraction converts this finite residual bound into a bound for the actual origin solution. Explicit constants are needed in our computer-assisted argument. In the following lemma we set $L_0=\left|\Omega+\frac{i}{p-1}\right|+p(1/2)^{p-1}$, which will be used as a Lipschitz bound for the full right-hand side in a neighborhood of the origin. We will close a fixed point argument in the ball of radius
\begin{equation*}
	\delta_{\mathrm{in}}=
	\frac{2r_{\mathrm{in}}^2H_0}{2d-r_{\mathrm{in}}^2L_0}.
\end{equation*}
\begin{lem}\label{lem:origin-contraction}
	Suppose that $|A| \leq 1/2$ and set $r_{\mathrm{in}}$ small enough so that
	\begin{equation}\label{eq:pure-slow-origin-tests}
		\frac{r_{\mathrm{in}}^2L_0}{2d}<1,\qquad
		\sup_{[0,r_{\mathrm{in}}]}|P_{\mathrm{in}}^{(I)}|+\delta_{\mathrm{in}}<\frac12,
	\end{equation}
	then there is an exact solution $Q^{\mathrm{in}}$ of \eqref{eq:complex-profile-radial} with $Q^{\mathrm{in}}(0)=A$, $(Q^{\mathrm{in}})'(0)=0$, and
	\begin{equation}\label{eq:pure-slow-origin-enclosure}
		|Q^{\mathrm{in}}(r_{\mathrm{in}})-P_{\mathrm{in}}^{(I)}(r_{\mathrm{in}})|\leq\delta_{\mathrm{in}},\qquad
		|(Q^{\mathrm{in}})'(r_{\mathrm{in}})-(P_{\mathrm{in}}^{(I)})'(r_{\mathrm{in}})|
		\leq\frac{r_{\mathrm{in}}}{d}
		(H_0+L_0\delta_{\mathrm{in}}).
	\end{equation}
\end{lem}
\begin{proof}
	The inverse with zero data $h(0) = h'(0)=0$ of the ODE $h''+((d-1)/r+ir/2)h'=f$ is
	\begin{equation*}
		(T_0f)(r)=\int_0^r t^{1-d}e^{-\frac{it^2}{4}}
		\int_0^t \sigma^{d-1}e^{\frac{i\sigma^2}{4}}f(\sigma)\,d\sigma\,dt.
	\end{equation*}
	It satisfies
	\begin{equation*}
		|(T_0f)'(r_{\mathrm{in}})|\leq\frac{r_{\mathrm{in}}}{d}\|f\|_\infty
		\qquad \text{ so } \qquad
		\|T_0f\|_\infty\leq\frac{r_{\mathrm{in}}^2}{2d}\|f\|_\infty.
	\end{equation*}
	The definition of $\delta_{\mathrm{in}}$ and the first inequality in \eqref{eq:pure-slow-origin-tests} give
	\begin{equation*}
		\frac{r_{\mathrm{in}}^2H_0}{2d}
		+\frac{r_{\mathrm{in}}^2L_0}{2d}\delta_{\mathrm{in}}\leq\delta_{\mathrm{in}}.
	\end{equation*}
	Hence the map
	\begin{equation*}
		h\longmapsto
		T_0\{N(P_{\mathrm{in}}^{(I)}+h)-N(P_{\mathrm{in}}^{(I)})
		-\left(\Omega+\frac{i}{p-1}\right)h-\mathcal R_{\mathrm{in}}\}
	\end{equation*}
	contracts on $\|h\|_\infty\leq\delta_{\mathrm{in}}$. Its fixed point gives the regular-origin solution from \eqref{eq:profile-local-integral}, by uniqueness, and the bounds for $T_0$ give \eqref{eq:pure-slow-origin-enclosure}. The same uniqueness identifies the constructed function with the smooth profile.
\end{proof}

\subsection{Matching the inner and outer solutions}

This subsection sets up the computer-assisted matching argument by comparing the exact inner and outer solutions selected at the origin and at infinity. We define the corresponding matching maps and state the criterion that will transfer rigorous finite computations to the exact problem.

For the computer-assisted argument, we set the parameters $r_{\mathrm{in}},r_{\mathrm{m}},r_{\mathrm{out}}$, which represent an inner, matching and outer radius respectively. We use $I,J$ for truncation orders of Taylor polynomials and exterior expansions respectively. These parameters are fixed in the code. Recall the parameter cube is defined by \eqref{eq:pure-slow-parameter-box}.

For each $\vartheta=(A,\Omega,K)$, we compare two solutions on each shooting interval. The solutions $Q^{\mathrm{in}}$ and $Q^{\mathrm{out}}$ are selected by the desired exact boundary conditions $Q^{\mathrm{in}}(0)=A$, $(Q^{\mathrm{in}})'(0)=0$, while $Q^{\mathrm{out}}$ has $C_f=0$ and $C_s=K^{1/(p-1)}$. Meanwhile $\widetilde Q^{\mathrm{in}}$ and $\widetilde Q^{\mathrm{out}}$ instead start from computable data. Specifically,
\begin{equation*}
	(\widetilde Q^{\mathrm{in}},(\widetilde Q^{\mathrm{in}})')(r_{\mathrm{in}})
	=(P_{\mathrm{in}}^{(I)},(P_{\mathrm{in}}^{(I)})')(r_{\mathrm{in}}),
	\qquad
	(\widetilde Q^{\mathrm{out}},(\widetilde Q^{\mathrm{out}})')(r_{\mathrm{out}})
	=(P_{\mathrm{out}}^{(J)},(P_{\mathrm{out}}^{(J)})')(r_{\mathrm{out}}).
\end{equation*}
All four functions solve the same profile ODE. The computable solutions $\widetilde Q^{\mathrm{in}}$ and $\widetilde Q^\mathrm{out}$ are the ones that we consider in the computer-assisted argument while the true solutions $Q^{\mathrm{in}}$ and $Q^{\mathrm{out}}$ are not explicit, as they exist only by the fixed point arguments above. We bound their difference using estimates from the contraction arguments.

For $\sigma\in\{\mathrm{in},\mathrm{out}\}$, define the exact and finite-data profile state vectors
\begin{equation*}
	\mathbf Q_\vartheta^\sigma(r)
	:=\begin{pmatrix}Q_\vartheta^\sigma(r)\\(Q_\vartheta^\sigma)'(r)\end{pmatrix},
	\qquad
	\widetilde{\mathbf Q}_\vartheta^\sigma(r)
	:=\begin{pmatrix}\widetilde Q_\vartheta^\sigma(r)\\(\widetilde Q_\vartheta^\sigma)'(r)\end{pmatrix}.
\end{equation*}
Where these solutions exist and are nonzero, set
\begin{equation*}
	\mathcal H\begin{pmatrix}q\\v\end{pmatrix}:=
	\begin{pmatrix}
		\log|q| \\ r_{\mathrm{m}}\operatorname{Re}(v/q)\\
		r_{\mathrm{m}}\operatorname{Im}(v/q)
	\end{pmatrix},
\end{equation*}
and define for $\vartheta=(A,\Omega,K)$ the matching maps
\begin{equation}\label{eq:pure-slow-matching-map}
	\begin{aligned}
		\mathcal M(\vartheta)
		 & :=\mathcal H(\mathbf Q_\vartheta^{\mathrm{in}})
		-\mathcal H(\mathbf Q_\vartheta^{\mathrm{out}}),               \\
		\widetilde{\mathcal M}(\vartheta)
		 & :=\mathcal H(\widetilde{\mathbf Q}_\vartheta^{\mathrm{in}})
		-\mathcal H(\widetilde{\mathbf Q}_\vartheta^{\mathrm{out}}),
	\end{aligned}
	\qquad\text{at }r=r_{\mathrm{m}}.
\end{equation}
These maps measure the mismatch between the inner and outer solutions at the matching radius $r_{\mathrm{m}}$. Logarithmic coordinates are chosen so that the difference is a ratio as opposed to an absolute difference. The factor $r_{\mathrm{m}}$ makes $r_{\mathrm{m}} v/q$ the derivative of $\log q$ with respect to $\log r$, whose real and imaginary parts are on a dimensionless scale. A zero of the exact map $\mathcal{M}$ is precisely the condition needed to join the two solutions, while the finite map $\widetilde{\mathcal{M}}$ is the computable object whose value and derivative are used in the computer-assisted existence argument.

The map $\mathcal M$ is continuous by Lemma~\ref{lem:profile-local-regularity} and Corollary~\ref{cor:pure-slow-exterior-family}, together with existence and nonvanishing on both shooting intervals from Lemma~\ref{lem:pure-slow-computer-bounds}. The initial data $P_{\mathrm{in}}^{(I)}$, $P_{\mathrm{out}}^{(J)}$ depend real analytically on the parameters. By $C^1$ dependence of the ODE flow on initial data and parameters, and smoothness of $\mathcal H$ for $q\ne0$, the map $\widetilde{\mathcal M}$ is $C^1$ on its domain of regularity. Its derivative is computed using
\begin{equation}\label{eq:pure-slow-matching-derivative}
	D\mathcal H\begin{pmatrix}q\\v\end{pmatrix}
	\begin{pmatrix}u\\w\end{pmatrix}=
	\begin{pmatrix}
		\operatorname{Re}(u/q)                      \\
		r_{\mathrm{m}}\operatorname{Re}(w/q-vu/q^2) \\
		r_{\mathrm{m}}\operatorname{Im}(w/q-vu/q^2)
	\end{pmatrix}.
\end{equation}

If $\mathcal M(\vartheta)=0$, a constant unit phase makes the inner and outer Cauchy data agree at $r_{\mathrm{m}}$. Then ODE uniqueness joins them into the desired global profile without changing $C_f=0$ or $|C_s|=K^{1/(p-1)}$. The next proposition establishes a criterion for the existence of such a zero.

\begin{prop}\label{prop:pure-slow-validated-zero}
	Suppose $\mathcal M$ and $\widetilde{\mathcal M}$ are defined on $D$,
	with the regularity stated above and radius $\rho$ as in
	\eqref{eq:pure-slow-parameter-box}. Let $B$ be an invertible real
	$3\times3$ matrix. If
	\begin{equation*}
		\begin{aligned}
			\|B\widetilde{\mathcal M}(\vartheta_0)\|_\infty       & \leq Y,       \\
			\sup_D\|B(\mathcal M-\widetilde{\mathcal M})\|_\infty & \leq E,       \\
			\sup_D\|I-BD\widetilde{\mathcal M}\|_\infty           & \leq Z,\qquad
			Y+E+Z\rho<\rho,
		\end{aligned}
	\end{equation*}
	then $\mathcal M$ has a zero in $D$. The matrix norm is the operator norm induced by the vector infinity norm.
\end{prop}
\begin{proof}
	The continuous map $T(\vartheta)=\vartheta-B\mathcal M(\vartheta)$ satisfies
	\begin{equation*}
		\begin{aligned}
			\|T(\vartheta)-\vartheta_0\|_\infty
			 & \leq \|B\widetilde{\mathcal M}(\vartheta_0)\|_\infty
			+Z\|\vartheta-\vartheta_0\|_\infty
			+\|B(\mathcal M-\widetilde{\mathcal M})(\vartheta)\|_\infty \\
			 & \leq Y+Z\rho+E<\rho.
		\end{aligned}
	\end{equation*}
	The middle term follows from convexity of $D$ and
	\begin{equation*}
		\begin{aligned}
			 & \bigl\|\vartheta-\vartheta_0
			-B\bigl(\widetilde{\mathcal M}(\vartheta)
			-\widetilde{\mathcal M}(\vartheta_0)\bigr)\bigr\|_\infty
			=\left\|\int_0^1
			\bigl[I-BD\widetilde{\mathcal M}(\vartheta_0+t(\vartheta-\vartheta_0))\bigr]
			(\vartheta-\vartheta_0)\,dt\right\|_\infty      \\
			 & \quad\leq Z\|\vartheta-\vartheta_0\|_\infty.
		\end{aligned}
	\end{equation*}
	Thus $T$ maps the compact convex cube $D$ into itself. Brouwer's fixed-point theorem gives a fixed point, and invertibility of $B$ makes it a zero of $\mathcal M$.
\end{proof}

\subsection{Computer-assisted proof of the matching conditions}
We check the inner shooting from $r_{\mathrm{in}}$ to $r_{\mathrm{m}}$ and the outer shooting from $r_{\mathrm{out}}$ back to $r_{\mathrm{m}}$ by the same stepwise estimates. On each mesh step, we have four polynomials that approximate the solution and its three derivatives with respect to $A, \Omega, K$.

Because $\widetilde{\mathcal M}$ is defined by the finite shooting solutions at $r_{\mathrm{m}}$, its value and derivative are bounded by propagating the Cauchy data through the profile equation. The profile ODE can be written as a first-order system for $(q,v)=(Q,Q')$ using the vector field
\begin{equation}\label{eq:pure-slow-first-order-system}
	f_\Omega\left(r,\begin{pmatrix}q\\v\end{pmatrix}\right)
	=\begin{pmatrix}
		v \\
		N(q)-\left(\Omega+\frac{i}{p-1}\right)q
		-\left(\frac{d-1}r+\frac{ir}{2}\right)v
	\end{pmatrix}.
\end{equation}
We write $D f_\Omega$ for the derivative with respect to $(q,v)$. On either mesh, we suppress the label $\sigma \in \left\{ \mathrm{in}, \mathrm{out} \right\}$ on $\mathbf Q_\vartheta^\sigma =(Q_\vartheta^\sigma,(Q_\vartheta^\sigma)')$ and $\widetilde{\mathbf Q}_\vartheta^\sigma =(\widetilde Q_\vartheta^\sigma,(\widetilde Q_\vartheta^\sigma)')$. The exact and finite-data solutions satisfy
\begin{equation*}
	(\mathbf Q_\vartheta)'=f_\Omega(r,\mathbf Q_\vartheta),
	\qquad
	(\widetilde{\mathbf Q}_\vartheta)'
	=f_\Omega(r,\widetilde{\mathbf Q}_\vartheta).
\end{equation*}

To bound $D\widetilde{\mathcal M}$, differentiate $\widetilde{\mathbf Q}_\vartheta$ with respect to its three real parameters. This gives the three variational equations
\begin{equation*}
	\begin{aligned}
		(\partial_A\widetilde{\mathbf Q}_\vartheta)'
		 & =D f_\Omega(r,\widetilde{\mathbf Q}_\vartheta)
		\partial_A\widetilde{\mathbf Q}_\vartheta,                           \\
		(\partial_\Omega\widetilde{\mathbf Q}_\vartheta)'
		 & =D f_\Omega(r,\widetilde{\mathbf Q}_\vartheta)
		\partial_\Omega\widetilde{\mathbf Q}_\vartheta
		-\begin{pmatrix}0\\(\widetilde{\mathbf Q}_\vartheta)_1\end{pmatrix}, \\
		(\partial_K\widetilde{\mathbf Q}_\vartheta)'
		 & =D f_\Omega(r,\widetilde{\mathbf Q}_\vartheta)
		\partial_K\widetilde{\mathbf Q}_\vartheta.
	\end{aligned}
\end{equation*}
Finally, the propagation estimates require bounds on the Jacobian and its variation. We use the Euclidean norms on $\mathbb C$ and $\mathbb C^2$, viewed as real vector spaces, and their induced operator norms for derivatives. Direct differentiation of $q^{(p+1)/2}\overline q^{\,(p-1)/2}$ gives
\begin{equation*}
	\begin{gathered}
		DN(q)u=\frac{p+1}{2}|q|^{p-1}u+\frac{p-1}{2}q^2|q|^{p-3}\overline u,\\
		\|DN(q)\|\leq p|q|^{p-1},\qquad
		\|D^2N(q)\|\leq p(p-1)|q|^{p-2}.
	\end{gathered}
\end{equation*}

We now explain how to propagate solutions and bound the error along a mesh step. Consider a mesh step $r(u)=r_*+hu$, $0\leq u\leq1$, where $h\ne0$ may have either sign, and set $r_{\min}=\min\{r_*,r_*+h\}>0$. Let $P(u)$ approximate $\widetilde{\mathbf Q}_{\vartheta_0}(r(u))$, and write $p_q$ for its first component. For each parameter label $\xi\in\{A,\Omega,K\}$, let $P_\xi(u)$ approximate $\left.\partial_\xi\widetilde{\mathbf Q}_\vartheta(r(u))\right|_{\vartheta=\vartheta_0}$. Primes on these polynomials denote differentiation with respect to $u$. Their integrated residuals satisfy
\begin{equation}\label{eq:pure-slow-step-defects}
	\begin{aligned}
		\int_0^1 \bigl\|P'-h f_{\Omega_0}(r(u),P)\bigr\|\,du
		 & \leq \delta_0,      \\
		\int_0^1 \bigl\|P_A'-h D f_{\Omega_0}(r(u),P)P_A\bigr\|\,du
		 & \leq \delta_A,      \\
		\int_0^1 \bigl\|P_\Omega'
		-h\{D f_{\Omega_0}(r(u),P)P_\Omega-(0,p_q)\}\bigr\|\,du
		 & \leq \delta_\Omega, \\
		\int_0^1 \bigl\|P_K'-h D f_{\Omega_0}(r(u),P)P_K\bigr\|\,du
		 & \leq \delta_K.
	\end{aligned}
\end{equation}
These bounds can be computed through finite polynomial calculations.

At the starting point $r_*$ of a step, suppose we have error bounds
\begin{equation}\label{eq:pure-slow-step-initial-errors}
	\begin{gathered}
		\|\widetilde{\mathbf Q}_{\vartheta_0}(r_*)-P(0)\|
		\leq\varepsilon_{\mathrm{poly}},\\
		\sup_{\vartheta\in D}
		\|\widetilde{\mathbf Q}_\vartheta(r_*)
		-\widetilde{\mathbf Q}_{\vartheta_0}(r_*)\|\leq r_Q,\\
		\sup_{\vartheta\in D}
		\|\mathbf Q_\vartheta(r_*)-\widetilde{\mathbf Q}_\vartheta(r_*)\|
		\leq \delta_{\mathrm{bd}}.
	\end{gathered}
\end{equation}
These bound, respectively, the polynomial error, the variation across the parameter set $D$, and the discrepancy between the exact and finite-data flows. For the parameter derivatives, also suppose
\begin{equation}\label{eq:pure-slow-step-initial-derivative-errors}
	\bigl\|\partial_\xi\widetilde{\mathbf Q}_\vartheta(r_*)\bigr|_{\vartheta=\vartheta_0}
	-P_\xi(0)\bigr\|
	\leq\varepsilon_{\mathrm{poly},\xi}, \qquad
	\sup_{\vartheta\in D}
	\bigl\|\partial_\xi\widetilde{\mathbf Q}_\vartheta(r_*)
	-\partial_\xi\widetilde{\mathbf Q}_\vartheta(r_*)\bigr|_{\vartheta=\vartheta_0}\bigr\|
	\leq r_\xi.
\end{equation}
The derivative bounds are needed to control $D\widetilde{\mathcal M}$ uniformly on $D$. The bounds in \eqref{eq:pure-slow-step-initial-errors}-- \eqref{eq:pure-slow-step-initial-derivative-errors} are checked at the first mesh point by interval arithmetic, with $\delta_{\mathrm{bd}}$ supplied by the origin or exterior boundary contraction argument. On later steps, they follow from the preceding bounds after new error is included.

We denote components of $P$ by $P(u) = (p_q(u), p_v(u))$. Choose an enclosure radius $\varepsilon>0$ and bounds
\begin{equation}\label{eq:pure-slow-step-enclosure-bounds}
	M_{\mathrm{step}}\geq\sup_{u\in[0,1]}|p_q(u)|+\varepsilon,
	\qquad
	V_\xi\geq\sup_{u\in[0,1]}\|P_\xi(u)\|.
\end{equation}
Thus $M_{\mathrm{step}}$ bounds $|q|$ for states within $\varepsilon$ of $P$, while $V_\xi$ bounds the polynomial parameter derivative. Define the growth factor $g$ for perturbations within this enclosure by
\begin{equation*}
	L=
	\frac{1+\sup_D|\Omega|+\frac{1}{p-1}+pM_{\mathrm{step}}^{p-1}}2
	+\begin{cases}0,&h>0,\\(d-1)/r_{\min},&h<0,\end{cases}
	\qquad g=\exp(|h|L).
\end{equation*}
We denote by a subscript $+$ a candidate bound valid throughout the current step, including its endpoint. Define
\begin{equation}\label{eq:pure-slow-propagation-tests}
	\begin{aligned}
		(\varepsilon_{\mathrm{poly}})_+ & =g\bigl(\varepsilon_{\mathrm{poly}}+\delta_0\bigr), \\
		(r_Q)_+                         & =g\bigl(r_Q+|h|\rho M_{\mathrm{step}}\bigr),        \\
		(\delta_{\mathrm{bd}})_+        & =g\delta_{\mathrm{bd}}.
	\end{aligned}
\end{equation}
These will be used as bounds in Lemma~\ref{lem:pure-slow-propagation}. These errors obey
\begin{equation*}
	\begin{aligned}
		(\varepsilon_{\mathrm{poly},\xi})_+
		 & =g\Bigl\{\varepsilon_{\mathrm{poly},\xi}+\delta_\xi
		+p(p-1)|h|M_{\mathrm{step}}^{p-2}V_\xi
		(\varepsilon_{\mathrm{poly}})_+\Bigr\},
		 &                                                           & \xi=A,K, \\
		(\varepsilon_{\mathrm{poly},\Omega})_+
		 & =g\Bigl\{\varepsilon_{\mathrm{poly},\Omega}+\delta_\Omega
		+|h|\bigl(p(p-1)M_{\mathrm{step}}^{p-2}V_\Omega+1\bigr)
		(\varepsilon_{\mathrm{poly}})_+\Bigr\}.
	\end{aligned}
\end{equation*}
Again, this is proven in Lemma~\ref{lem:pure-slow-propagation}. For their variation across the parameter cube, set
\begin{equation*}
	\begin{aligned}
		(r_\xi)_+
		 & =g\Bigl\{r_\xi+|h|(p(p-1)M_{\mathrm{step}}^{p-2}(r_Q)_++\rho)
		\bigl(V_\xi+(\varepsilon_{\mathrm{poly},\xi})_+\bigr)\Bigr\},
		 &                                                                  & \xi=A,K, \\
		(r_\Omega)_+
		 & =g\Bigl\{r_\Omega+|h|(p(p-1)M_{\mathrm{step}}^{p-2}(r_Q)_++\rho)
		\bigl(V_\Omega+(\varepsilon_{\mathrm{poly},\Omega})_+\bigr)
		+|h|(r_Q)_+\Bigr\}.
	\end{aligned}
\end{equation*}
The nonvanishing test for the step is
\begin{equation}\label{eq:pure-slow-step-nonvanishing}
	|p_q(0)|-\sum_{n\geq1}|[u^n]p_q|>\varepsilon.
\end{equation}
Here $[u^n]p_q$ is the coefficient of $u^n$ in $p_q$. Since $0\leq u\leq1$, the left-hand side bounds $|p_q(u)|$ from below throughout the step. Once the enclosure condition in the next lemma holds, every exact and finite profile lies within $\varepsilon$ of $P$ and therefore has $|q|>0$. This also ensures that the logarithm and ratios in the matching map are defined at $r_{\mathrm{m}}$.

The next lemma turns the residual and starting bounds into enclosures valid throughout one step. Its strict inequality ensures that the solutions stay inside the enclosure used to compute $g$.

\begin{lem}\label{lem:pure-slow-propagation}
	Assume the residual bounds \eqref{eq:pure-slow-step-defects},
	the starting bounds
	\eqref{eq:pure-slow-step-initial-errors}--
	\eqref{eq:pure-slow-step-initial-derivative-errors},
	and the enclosure bounds \eqref{eq:pure-slow-step-enclosure-bounds} defining
	$M_{\mathrm{step}}$ and $V_\xi$.
	If
	\begin{equation}\label{eq:pure-slow-step-enclosure-condition}
		(\varepsilon_{\mathrm{poly}})_++(r_Q)_++(\delta_{\mathrm{bd}})_+<\varepsilon,
	\end{equation}
	then the following bounds hold for every $u\in[0,1]$
	\begin{equation*}
		\begin{gathered}
			\|\widetilde{\mathbf Q}_{\vartheta_0}-P\|
			\leq(\varepsilon_{\mathrm{poly}})_+,
			\qquad
			\sup_{\vartheta\in D}\|\widetilde{\mathbf Q}_\vartheta
			-\widetilde{\mathbf Q}_{\vartheta_0}\|\leq (r_Q)_+,
			\qquad
			\sup_{\vartheta\in D}\|\mathbf Q_\vartheta
			-\widetilde{\mathbf Q}_\vartheta\|
			\leq (\delta_{\mathrm{bd}})_+,\\
			\|\left.\partial_\xi\widetilde{\mathbf Q}_\vartheta
			\right|_{\vartheta=\vartheta_0}-P_\xi\|
			\leq(\varepsilon_{\mathrm{poly},\xi})_+,\\
			\sup_{\vartheta\in D}\|\partial_\xi\widetilde{\mathbf Q}_\vartheta
			-\left.\partial_\xi\widetilde{\mathbf Q}_\vartheta\right|_{\vartheta=\vartheta_0}\|
			\leq(r_\xi)_+,
			\quad \xi\in\{A,\Omega,K\}.
		\end{gathered}
	\end{equation*}
	All exact and finite profile solutions involved continue to the endpoint. If, in addition, \eqref{eq:pure-slow-step-nonvanishing} holds, they are nonvanishing throughout the step.
\end{lem}
\begin{proof}
	We first bound the growth of perturbations to the profile flow. For every $(\eta,\zeta)\in\mathbb C^2$, direct differentiation of \eqref{eq:pure-slow-first-order-system} gives
	\begin{equation*}
		D f_\Omega(r,(q,v))(\eta,\zeta)
		=\left(\zeta,DN(q)\eta-
		\left(\Omega+\frac{i}{p-1}\right)\eta
		-\left(\frac{d-1}{r}+\frac{ir}{2}\right)\zeta\right).
	\end{equation*}
	If $|q|\leq M_{\mathrm{step}}$, then $\|DN(q)\|\leq pM_{\mathrm{step}}^{p-1}$. Moreover, the term $-ir\zeta/2$ has zero real inner product with $\zeta$. The radial damping decreases the norm when $h>0$, whereas for $h<0$ its contribution is at most $(d-1)/r_{\min}$. Consequently,
	\begin{equation}\label{eq:pure-slow-secant-estimate}
		\begin{aligned}
			\operatorname{Re}\left\langle(\eta,\zeta),
			\frac{h}{|h|}D f_\Omega(r,(q,v))(\eta,\zeta)\right\rangle
			 & \leq
			\left(1+|\Omega|+\frac1{p-1}+pM_{\mathrm{step}}^{p-1}\right)
			|\eta||\zeta|                                    \\
			 & \quad+
			\frac{d-1}{r_{\min}}\mathbf 1_{\{h<0\}}|\zeta|^2 \\
			 & \leq L\bigl(|\eta|^2+|\zeta|^2\bigr).
		\end{aligned}
	\end{equation}

	We give the details for the polynomial error. Set
	\begin{equation*}
		\mathcal E(u)=\widetilde{\mathbf Q}_{\vartheta_0}(r(u))-P(u),
		\qquad
		\mathcal R(u)=P'(u)-h f_{\Omega_0}(r(u),P(u)).
	\end{equation*}
	The fundamental theorem of calculus applied to the vector field yields
	\begin{equation*}
		\mathcal E'(u)
		=h\int_0^1D f_{\Omega_0}
		\bigl(r(u),P(u)+\tau\mathcal E(u)\bigr)
		\mathcal E(u)\,d\tau-\mathcal R(u).
	\end{equation*}
	Thus \eqref{eq:pure-slow-secant-estimate} implies, wherever $\mathcal E(u)\ne0$,
	\begin{equation*}
		\frac{d}{du}\|\mathcal E(u)\|
		=\frac{\operatorname{Re}\langle\mathcal E(u),\mathcal E'(u)\rangle}
		{\|\mathcal E(u)\|}
		\leq |h|L\|\mathcal E(u)\|+\|\mathcal R(u)\|.
	\end{equation*}
	Multiplying by $e^{-|h|Lu}$ and using \eqref{eq:pure-slow-step-defects}, for $0\leq u\leq1$ we obtain
	\begin{equation*}
		\begin{aligned}
			\|\mathcal E(u)\|
			 & \leq e^{|h|Lu}\|\mathcal E(0)\|
			+\int_0^u e^{|h|L(u-t)}\|\mathcal R(t)\|\,dt       \\
			 & \leq e^{|h|Lu}\left(\varepsilon_{\mathrm{poly}}
			+\int_0^u\|\mathcal R(t)\|\,dt\right)
			\leq g\bigl(\varepsilon_{\mathrm{poly}}+\delta_0\bigr).
		\end{aligned}
	\end{equation*}
	This proves the first line of \eqref{eq:pure-slow-propagation-tests}. For $\widetilde{\mathbf Q}_\vartheta -\widetilde{\mathbf Q}_{\vartheta_0}$, the additional source $\bigl(0,-(\Omega-\Omega_0)(\widetilde{\mathbf Q}_{\vartheta_0})_1\bigr)$ has norm at most $\rho M_{\mathrm{step}}$, giving the second line. The exact and finite solutions at fixed parameter solve the same ODE, giving the third line.

	Subtracting the center variational equation from its polynomial approximation gives an additional source bounded by $p(p-1)M_{\mathrm{step}}^{p-2}V_\xi (\varepsilon_{\mathrm{poly}})_+$, together with $(\varepsilon_{\mathrm{poly}})_+$ for the frequency derivative, proving the bound for $(\varepsilon_{\mathrm{poly},\xi})_+$. For the variation across the parameter cube,
	\begin{equation*}
		\|D f_\Omega(r,\widetilde{\mathbf Q}_\vartheta)
		-D f_{\Omega_0}(r,\widetilde{\mathbf Q}_{\vartheta_0})\|
		\leq p(p-1)M_{\mathrm{step}}^{p-2}(r_Q)_++\rho.
	\end{equation*}
	Also $\|\left.\partial_\xi\widetilde{\mathbf Q}_\vartheta \right|_{\vartheta=\vartheta_0}\| \leq V_\xi+(\varepsilon_{\mathrm{poly},\xi})_+$ throughout the step. The difference of the frequency sources is bounded by $(r_Q)_+$ when $\xi=\Omega$ and vanishes when $\xi=A,K$. Subtracting the variational equations and applying Gronwall, including the additional source bounded by $|h|(r_Q)_+$ when $\xi=\Omega$, proves the last estimate.

	These estimates are initially valid up to a possible first exit from the enclosure. Their strict sum bound excludes such an exit. The vector field is smooth for $r>0$, and bounded solutions continue, proving existence through the step. Finally, \eqref{eq:pure-slow-step-nonvanishing} and the enclosure bound give a positive lower bound for the modulus.
\end{proof}

Lemma~\ref{lem:pure-slow-propagation} is applied successively on the mesh covering the inner and outer finite intervals. At each step the polynomial coefficients are exact numbers. Errors are added to the propagated bounds before the next step, and the enclosure condition is rechecked. The initial bounds and parameter variations are obtained by interval evaluation and differentiation of the finite formulas \eqref{eq:pure-slow-origin-recursion} and \eqref{eq:pure-slow-E-recursion}--\eqref{eq:pure-slow-ell-recursion}.

It remains to convert the endpoint enclosure into an error bound for the matching map. Suppose a finite endpoint satisfies $|q|\geq q_{\min}$ and $|v|\leq v_{\max}$, while the exact endpoint is within Euclidean distance $\delta_{\mathrm{bd}}<q_{\min}$. For the inner endpoints, the amplitude and slope bounds are
\begin{equation*}
	\begin{aligned}
		\left|\log|Q^{\mathrm{in}}|-\log|\widetilde Q^{\mathrm{in}}|\right|
		 & \leq \frac{\delta_{\mathrm{bd}}}{q_{\min}-\delta_{\mathrm{bd}}},                                     \\
		r_{\mathrm{m}}\left|\operatorname{Re}\left(
		\frac{(Q^{\mathrm{in}})'}{Q^{\mathrm{in}}}
		-\frac{(\widetilde Q^{\mathrm{in}})'}{\widetilde Q^{\mathrm{in}}}
		\right)\right|
		 & \leq \frac{r_{\mathrm{m}} \delta_{\mathrm{bd}}(1+v_{\max}/q_{\min})}{q_{\min}-\delta_{\mathrm{bd}}}, \\
		r_{\mathrm{m}}\left|\operatorname{Im}\left(
		\frac{(Q^{\mathrm{in}})'}{Q^{\mathrm{in}}}
		-\frac{(\widetilde Q^{\mathrm{in}})'}{\widetilde Q^{\mathrm{in}}}
		\right)\right|
		 & \leq \frac{r_{\mathrm{m}} \delta_{\mathrm{bd}}(1+v_{\max}/q_{\min})}{q_{\min}-\delta_{\mathrm{bd}}},
	\end{aligned}
	\qquad\text{at }r=r_{\mathrm{m}}.
\end{equation*}
The same inequalities hold for the outer endpoints with $\mathrm{in}$ replaced by $\mathrm{out}$. The first bound follows by integrating the derivative of $\log|q|$ along the segment between the endpoints. The second and third follow by subtracting the two ratios $v/q$. For $\sigma\in\{\mathrm{in},\mathrm{out}\}$, let $q_{\min}^\sigma,v_{\max}^\sigma,\delta_{\mathrm{bd}}^\sigma$ denote the corresponding endpoint bounds and define the vector
\begin{equation*}
	e^\sigma:= \left( \frac{\delta_{\mathrm{bd}}^\sigma}{q_{\min}^\sigma-\delta_{\mathrm{bd}}^\sigma},
	\frac{r_{\mathrm{m}} \delta_{\mathrm{bd}}^\sigma(1+v_{\max}^\sigma/q_{\min}^\sigma)}{q_{\min}^\sigma-\delta_{\mathrm{bd}}^\sigma},
	\dfrac{r_{\mathrm{m}} \delta_{\mathrm{bd}}^\sigma(1+v_{\max}^\sigma/q_{\min}^\sigma)}{q_{\min}^\sigma-\delta_{\mathrm{bd}}^\sigma}
	\right).
\end{equation*}
Thus, with $|B|$ denoting the entrywise absolute value of $B$, the choice
\begin{equation}\label{eq:pure-slow-matching-E}
	E:=\bigl\||B|(e^{\mathrm{in}}+e^{\mathrm{out}})\bigr\|_\infty
\end{equation}
gives a computable bound $\sup_D\|B(\mathcal M-\widetilde{\mathcal M})\|_\infty\leq E$. The next lemma states the computer-assisted bounds.

\begin{lem}[Computer-assisted bounds]\label{lem:pure-slow-computer-bounds}
	For the cube $D$ in \eqref{eq:pure-slow-parameter-box}, the following statements hold.
	\begin{enumerate}[label=\textup{(\roman*)}]
		\item Throughout $D$, we have $|\Omega|<\Omega_{\mathrm{bd}}$ and $K>0$. Moreover, the fixed-point conditions at the origin \eqref{eq:pure-slow-origin-tests} and in the exterior \eqref{eq:pure-slow-exterior-tests} hold uniformly, and the corresponding ball inclusions are strict.
		\item The bounds in \eqref{eq:pure-slow-propagation-tests} satisfy $(\varepsilon_{\mathrm{poly}})_++(r_Q)_++(\delta_{\mathrm{bd}})_+<\varepsilon$ on every step, with $\varepsilon=10^{-8}$, and \eqref{eq:pure-slow-step-nonvanishing} holds on every step. Thus the exact and finite inner solutions exist and are nonvanishing on $[r_{\mathrm{in}},r_{\mathrm{m}}]$, and the exact and finite outer solutions do so on $[r_{\mathrm{m}},r_{\mathrm{out}}]$, for all $\vartheta\in D$.
		\item The matrix $B$ is invertible and, for the bounds in Proposition~\ref{prop:pure-slow-validated-zero}, the quantities $Y,E,Z$ satisfy
		      \begin{equation*}
			      Y+E+Z\rho<\rho.
		      \end{equation*}
	\end{enumerate}
\end{lem}
\begin{proof}
	All the code used for the computer-assisted argument is available in the GitHub repository~\cite{GITHUB}.

	The first step is to generate approximate profiles. A broad exploration in the parameters $(d, p, A, \Omega, K)$ was carried out to look for solutions of the ODE \eqref{eq:complex-profile-radial}. Once several candidates had been found, the parameters were refined to improve precision. This is done by \path{generate_candidate.py} using numerical shooting and continuation. The profile ODE is easier to solve numerically for larger values of $p$, so an initial solution is first found in that regime and then continued to the desired $d$ and down to the final value of $p$, using each solution to initialize the next. The resulting candidate is refined to get low residuals using multiprecision integration with \texttt{mpmath}~\cite{mpmath} and Newton corrections. The result is an approximation $(A_0,\Omega_0,K_0)$ of the desired parameters.

	The rest of the code uses the library \texttt{python-flint} to implement interval arithmetic; see \cite{Johansson2017}. This is a Python module wrapping FLINT and Arb. A precision of 896 bits is chosen to ensure that error enclosures are small enough for the proof to close.

	The script \path{origin.py} evaluates the recurrence \eqref{eq:pure-slow-origin-recursion}, computes the residual estimate \eqref{eq:pure-slow-origin-residual-bound}, and checks the conditions for the contraction in \eqref{eq:pure-slow-origin-tests}. The file \path{exterior.py} evaluates the recurrence \eqref{eq:pure-slow-E-recursion}--\eqref{eq:pure-slow-ell-recursion}, computes the linear-mode bounds \eqref{eq:pure-slow-linear-mode-bounds}, computes the residual estimate \eqref{eq:pure-slow-residual-bound}, checks the restriction $K>0$, and all three conditions for the contraction in \eqref{eq:pure-slow-exterior-tests}. These computations give the initial Cauchy-data errors at $r_{\mathrm{in}}$ and $r_{\mathrm{out}}$, respectively, and prove item~\textup{(i)}.

	The script \path{propagation.py} starts from either $r_{\mathrm{in}}$ or $r_{\mathrm{out}}$. It constructs the coefficients of $P,P_A,P_\Omega,P_K$ successively from the formal Taylor identities
	\begin{equation*}
		\begin{aligned}
			P'        & =h f_{\Omega_0}(r(u),P),
			          & P_A'                                          & =hD f_{\Omega_0}(r(u),P)P_A, \\
			P_\Omega' & =h\{D f_{\Omega_0}(r(u),P)P_\Omega-(0,p_q)\},
			          & P_K'                                          & =hD f_{\Omega_0}(r(u),P)P_K.
		\end{aligned}
	\end{equation*}
	At the chosen Taylor degree the series are truncated and their coefficients are replaced by midpoint values. The program then substitutes the resulting polynomials into the four left-hand sides in \eqref{eq:pure-slow-step-defects} and obtains the bounds $\delta_0,\delta_A,\delta_\Omega,\delta_K$. It then applies Lemma~\ref{lem:pure-slow-propagation}, step by step, outward from $r_{\mathrm{in}}$ and inward from $r_{\mathrm{out}}$ to $r_{\mathrm{m}}$. At every endpoint, interval evaluation and errors are added to the bounds for the next step. On every step the program verifies \eqref{eq:pure-slow-step-enclosure-condition} with $\varepsilon=10^{-8}$ and the lower bound \eqref{eq:pure-slow-step-nonvanishing}. Thus all the hypotheses asserted in item~\textup{(ii)} are verified.

	Finally, \path{matching.py} evaluates $\widetilde{\mathcal M}$ from \eqref{eq:pure-slow-matching-map} and an interval enclosure of $D\widetilde{\mathcal M}$ using \eqref{eq:pure-slow-matching-derivative}. It takes $B$ to be the midpoint of an interval enclosure of the inverse of the Jacobian. Then it computes $Y$ and $Z$ directly and evaluates $E$ from \eqref{eq:pure-slow-matching-E}. The final strict check is precisely $Y+E+Z\rho<\rho$, proving item~\textup{(iii)}.

	The script \path{prove.py} runs all the above steps in order, and saves the results indicating whether each condition was verified.
\end{proof}

\begin{proof}[Proof of Theorem~\ref{thm:pure-slow-existence}]
	The bounds in Lemma~\ref{lem:pure-slow-computer-bounds} are used at several points in the argument. Item~\textup{(i)} verifies the hypotheses of Lemma~\ref{lem:origin-contraction} and of Proposition~\ref{prop:pure-slow-exterior-certificate}. Item~\textup{(ii)} provides the hypotheses of Lemma~\ref{lem:pure-slow-propagation} and the nonvanishing needed for the finite inner and outer solutions. Item~\textup{(iii)} supplies the three quantitative hypotheses of Proposition~\ref{prop:pure-slow-validated-zero}.

	It follows from the first two items and the cited results that the inner and outer solutions, together with their finite approximations, are defined up to the matching point. Lemma~\ref{lem:profile-local-regularity} and Corollary~\ref{cor:pure-slow-exterior-family} show that $\mathcal M$ is defined and continuous throughout $D$, while the corresponding ODE dependence shows that $\widetilde{\mathcal M}$ is differentiable there. The third item therefore allows us to apply Proposition~\ref{prop:pure-slow-validated-zero}, which gives a zero of $\mathcal M$ in $D$. Joining the inner and outer solutions by the constant phase described above produces a smooth global profile. It is regular at zero and has $C_f=0$ and $|C_s|=K^{1/(p-1)}>0$. Nonvanishing follows from identity \eqref{eq:profile-current}. Corollary~\ref{cor:pure-slow-full-expansion} gives the full slow expansion.
\end{proof}

\section{Self-similar blow-up}\label{sec:self-similar-blowup}

Fix the pure slow profile of Theorem~\ref{thm:pure-slow-existence}, and write
\begin{equation*}
	A=A_*,\qquad \Omega=\Omega_*,\qquad Q=Q_{A,\Omega}.
\end{equation*}
All functions and spaces in this section are radial, and constants may depend on the fixed profile. The construction below uses a finite-codimension strategy as in Li~\cite{Li2023}, together with the stable/unstable semigroup framework of~\cite{BuckmasterCaoLaboraGomezSerrano2025}. The evolution is studied around the stationary profile, and localization is implemented by a direct finite-dimensional correction to obtain a finite-energy solution.

In the coordinates of \eqref{eq:coordinates}, write
\begin{equation*}
	u(t,x)=(T-t)^{-\frac{1}{p-1}}e^{-i\Omega\tau}V(\tau,y),\qquad
	\tau=\log\frac{T}{T-t},\qquad y=\frac{x}{\sqrt{T-t}}.
\end{equation*}
The rescaled equation is
\begin{equation}\label{eq:complex-rescaled-evolution}
	\partial_\tau V=\mathcal A_0V-iN(V),\qquad
	\mathcal A_0=i\Delta-\frac12 y\cdot\nabla_y-\frac{1}{p-1}+i\Omega.
\end{equation}
The profile satisfies $\mathcal A_0Q=iN(Q)$. Hence $V=Q+v$ obeys
\begin{equation}\label{eq:exact-perturbation-system}
	\partial_\tau v=\mathcal A_Qv+\mathcal F(v),
\end{equation}
Set
\begin{equation}\label{eq:linearized-coefficients}
	V_{\mathrm{diag}}=\frac{p+1}{2}|Q|^{p-1},\qquad
	V_{\mathrm{off}}=\frac{p-1}{2}|Q|^{p-3}Q^2.
\end{equation}
Then the linearized operator $\mathcal{A}_Q$ and the nonlinear remainder $\mathcal F$ are
\begin{equation*}
	\begin{aligned}
		\mathcal A_Qh & =\mathcal A_0h-iV_{\mathrm{diag}}h-iV_{\mathrm{off}}\overline h, \\
		\mathcal F(v) & =-i\bigl[N(Q+v)-N(Q)
		-V_{\mathrm{diag}}v-V_{\mathrm{off}}\overline v\bigr].
	\end{aligned}
\end{equation*}
We first translate the asymptotics and nonvanishing of $Q$ into global pointwise bounds for the profile and its derivatives. These bounds place $Q$ in the perturbation spaces used below and ensure that cutting off its tail produces a small perturbation. They also show that the coefficients of the linearized operator decay quadratically, with an additional power of decay for each derivative, as needed for the subsequent spectral estimates.
\begin{lem}
	\label{lem:slow-profile-symbol-bounds}
	There are constants $c_*>0$ and $C_{Q,j}$ such that
	\begin{equation}\label{eq:slow-profile-symbol-bounds}
		|Q(y)|\geq c_*\langle y\rangle^{-\frac{2}{p-1}},\qquad
		|\partial^\nu Q(y)|\leq C_{Q,j}\langle y\rangle^{-\frac{2}{p-1}-j},
		\qquad |\nu|=j\geq0.
	\end{equation}
	In particular,
	\begin{equation}\label{eq:linearized-coefficient-derivative-bounds}
		|\partial^\nu V_{\mathrm{diag}}(y)|+|\partial^\nu V_{\mathrm{off}}(y)|
		\leq C_{Q,j}\langle y\rangle^{-2-j}.
	\end{equation}
\end{lem}
\begin{proof}
	Nonvanishing and the nonzero leading slow coefficient give the lower bound. Write $r=|y|$ and take the approximation $P_{\mathrm{out}}^{(J)}$ from Corollary~\ref{cor:pure-slow-full-expansion}, with the phase of $Q$, and put $w_J=Q-P_{\mathrm{out}}^{(J)}$. Its equation is
	\begin{equation*}
		\begin{aligned}
			w_J''={} & -\left(\frac{d-1}{r}+\frac{ir}{2}\right)w_J'
			-\left(\Omega+\frac{i}{p-1}\right)w_J                   \\
			         & +N(Q)-N(P_{\mathrm{out}}^{(J)})-r_J,         \\
			r_J={}   & L_\Omega P_{\mathrm{out}}^{(J)}
			-N(P_{\mathrm{out}}^{(J)}).
		\end{aligned}
	\end{equation*}
	The explicit residual satisfies $r_J^{(k)}=O(r^{-2/(p-1)-2J-2-k})$. Starting with the bounds for $w_J,w_J'$ in that corollary, differentiation of this equation gives $w_J^{(j)}=O(r^{-2/(p-1)-2J-4+j})$ for each fixed $j\geq1$, provided $J\geq j-2$. Increasing $J$ proves the radial derivative bounds; smoothness at zero and the formulas for Cartesian derivatives give \eqref{eq:slow-profile-symbol-bounds}. The coefficient bounds follow by the product rule.
\end{proof}

\subsection{Functional setting}\label{subsec:stability-spaces}

Recall that $s_c=d/2-2/(p-1)$. For $s_c<s<d/2$, let
\begin{equation*}
	\mathcal X_s=\dot H^s_{\mathrm{rad}}(\mathbb R^d;\mathbb C)
	\cap\dot H^{k_\star}_{\mathrm{rad}}(\mathbb R^d;\mathbb C),
	\qquad
	\|h\|_{\mathcal X_s}^2
	=\||D|^s h\|_2^2+\||D|^{k_\star} h\|_2^2,
\end{equation*}
where $k_\star\in \mathbb{N}$ satisfies
\begin{equation*}
	k_\star\ge3,\qquad
	k_\star>\frac d2+2\|V_{\mathrm{off}}\|_\infty.
\end{equation*}
This ensures $(s_c-k_\star)/2+\|V_{\mathrm{off}}\|_\infty\leq(s_c-s)/2$.

We use its underlying real Hilbert space for the scalar perturbation equation, since $\mathcal A_Q$ is real-linear. The $\dot H^s$ representative is fixed by the embedding into $L^{2d/(d-2s)}$. Equivalently, it is the completion of radial $C_c^\infty$ in the displayed norm. We record the details of the interpolation estimate that will be used below. For $h\in C_c^\infty$ and $R>0$, Fourier inversion and Cauchy--Schwarz give
\begin{equation} \label{eq:Linfty-embedding-proof}
	\begin{aligned}
		\|h\|_{L^\infty} \lesssim \|\widehat h\|_1
		 & \leq \int_{|\xi|\leq R}|\widehat h(\xi)|\,d\xi
		+\int_{|\xi|>R}|\widehat h(\xi)|\,d\xi            \\
		 & \leq
		\left(\int_{|\xi|\leq R}|\xi|^{-2s}\,d\xi\right)^{1/2}
		\|h\|_{\dot H^s}
		+\left(\int_{|\xi|>R}|\xi|^{-2k_{\star}}\,d\xi\right)^{1/2}
		\|h\|_{\dot H^{k_{\star}}}                        \\
		 & \lesssim_s R^{\frac d2-s}\|h\|_{\dot H^s}
		+R^{\frac d2-k_{\star}}\|h\|_{\dot H^{k_\star}}.
	\end{aligned}
\end{equation}
Choosing $R=(\|h\|_{\dot H^{k_{\star}}}/\|h\|_{\dot H^s})^{1/(k_{\star}-s)}$ balances the two terms and yields
\begin{equation}\label{eq:stability-Linfty-embedding}
	\|h\|_\infty
	\lesssim_s \|h\|_{\dot H^s}^{\frac{k_{\star}-\frac{d}{2}}{k_{\star}-s}}
	\|h\|_{\dot H^{k_{\star}}}^{\frac{\frac{d}{2}-s}{k_{\star}-s}}
	\lesssim_s\|h\|_{\mathcal X_s}.
\end{equation}
Here we used the inequality $A^{\theta}B^{1-\theta} \leq \theta A + (1-\theta)B$.

\begin{lem}\label{lem:profile-stability-space}
	For every $s_c<s<d/2$, the profile $Q$ belongs to $\mathcal X_s$.
\end{lem}
\begin{proof}
	Fix a smooth radial cutoff $\chi$, equal to one on $|y|\leq1$ and zero on $|y|\geq2$, and set $\chi_R(y)=\chi(y/R)$ and $\eta(y)=\chi(y)-\chi(2y)$. For every $R\geq1$,
	\begin{equation*}
		Q=\chi_RQ+\sum_{j\geq1}Q_{j,R},\qquad
		Q_{j,R}(y)=\eta(y/(2^jR))Q(y).
	\end{equation*}
	For $L=2^jR$, set $g_{j,R}(Y)=L^{2/(p-1)}\eta(Y)Q(LY)$. These functions are supported in $1/2\leq|Y|\leq2$, and \eqref{eq:slow-profile-symbol-bounds} gives $\|g_{j,R}\|_{H^{k_{\star}}}\leq C_Q$, uniformly in $j,R$. Scaling in dimension $d$ yields, for $k=s,k_{\star}$,
	\begin{equation*}
		\|Q_{j,R}\|_{\dot H^k}
		=L^{s_c-k}\|g_{j,R}\|_{\dot H^k}
		\leq C_Q(2^jR)^{s_c-k}.
	\end{equation*}
	Since $k>s_c$, the series converges in $\mathcal X_s$ and locally to $(1-\chi_R)Q$, with
	\begin{equation}\label{eq:initial-cutoff-smallness}
		\|(1-\chi_R)Q\|_{\dot H^k}
		\leq\sum_{j\geq1}\|Q_{j,R}\|_{\dot H^k}
		\lesssim_{Q,s}R^{s_c-k},\qquad k=s,k_{\star}.
	\end{equation}
	As $\chi_RQ\in C_c^\infty$, this also proves $Q\in\mathcal X_s$.
\end{proof}

Estimate~\eqref{eq:Linfty-embedding-proof} gives $\|\widehat h\|_1\lesssim_s\|h\|_{\mathcal X_s}$. Hence, for $k=s,k_\star$, the inequality $|\xi|^k\lesssim_k|\eta|^k+|\xi-\eta|^k$ and Young's convolution inequality give
\begin{equation*}
	\|fg\|_{\dot H^k}
	\lesssim_k \|f\|_{\dot H^k}\|\widehat g\|_1
	+\|\widehat f\|_1\|g\|_{\dot H^k}.
\end{equation*}
Thus $\mathcal X_s$ is an algebra:
\begin{equation}\label{eq:algebra-estimate}
	\|fg\|_{\mathcal X_s}
	\leq C_s\|f\|_{\mathcal X_s}\|g\|_{\mathcal X_s}.
\end{equation}
Since $p$ is odd, $\mathcal F(v)$ is a polynomial in $Q,\overline Q,v,\overline v$ whose terms have degree at least two in $v,\overline v$. Lemma~\ref{lem:profile-stability-space} and the algebra estimate therefore give
\begin{equation}\label{eq:stability-nonlinear-estimate}
	\|\mathcal F(v)-\mathcal F(w)\|_{\mathcal X_s}
	\leq C_{Q,s}
	\bigl(\|v\|_{\mathcal X_s}+\|w\|_{\mathcal X_s}+\|v\|^{p}_{\mathcal X_s}+\|w\|^{p}_{\mathcal X_s}\bigr)
	\|v-w\|_{\mathcal X_s}.
\end{equation}
\begin{lem}\label{lem:linearized-flow}
	$\mathcal A_{Q}$ generates a strongly continuous group $S_{Q}(\tau)$. Moreover, for any $v_0\in \mathcal X_s$, there exists $T>0$ such that a unique mild solution in $C([0,T];\mathcal X_s)$ solves \eqref{eq:exact-perturbation-system} with initial data $v_0$.
\end{lem}
\begin{proof}

    The unperturbed operator $\mathcal A_0$ generates the group $S_0(\tau) = e^{\tau \mathcal{A}_0}$, given by
	\begin{equation}\label{eq:free-rescaled-propagator}
		(S_0(\tau)h)(y)
		=e^{(-\frac{1}{p-1}+i\Omega)\tau}
		\bigl(e^{i(1-e^{-\tau})\Delta}h\bigr)(e^{-\frac{\tau}{2}}y),
		\qquad \tau\in\mathbb R.
	\end{equation}
	For $k=s,k_{\star}$, unitarity of the Schr\"odinger group on $\dot H^k$ and $\|h(\lambda\cdot)\|_{\dot H^k}=\lambda^{k-d/2}\|h\|_{\dot H^k}$ give
	\begin{equation*}
		\|S_0(\tau)h\|_{\dot H^k}
		=e^{\frac{(s_c-k)\tau}{2}}\|h\|_{\dot H^k},
	\end{equation*}
	since $s_c=d/2-2/(p-1)$. Thus we have for $\tau\geq 0$,
	\begin{equation*}
		\|S_0(\tau)h\|_{\mathcal{X}_s}
		\leq e^{\frac{(s_c-s)\tau}{2}}\|h\|_{\mathcal{X}_s}.
	\end{equation*}
	For radial $C_c^\infty$ data, continuity at $\tau=0$ follows from dominated convergence in Fourier space and continuity of dilations. Density and the preceding locally uniform operator bound extend this to every $h\in\mathcal X_s$, proving strong continuity.

	Moreover, $V_{\mathrm{diag}}$ and $V_{\mathrm{off}}$ belong to $\mathcal X_s$ by Lemma~\ref{lem:profile-stability-space} and the algebra estimate \eqref{eq:algebra-estimate}. Thus $\mathcal A_Q-\mathcal A_0$ is a bounded real-linear operator on $\mathcal X_s$. The bounded perturbation theorem yields a strongly continuous group $S_Q(\tau)=e^{\tau\mathcal A_Q}$, with $D(\mathcal A_Q)=D(\mathcal A_0)$. Now we show the local well-posedness of \eqref{eq:exact-perturbation-system} in $\mathcal X_s$. Using the Duhamel formula, we consider
	\begin{equation}\label{eq:duhamel}
		v(\tau)=S_Q(\tau)v_0+
		\int_0^\tau S_Q(\tau-\sigma)\mathcal F(v(\sigma))\,d\sigma.
	\end{equation}
	Estimate \eqref{eq:stability-nonlinear-estimate} makes $\mathcal F$ locally Lipschitz near zero. The local bound on $S_Q$ thus makes the equation \eqref{eq:duhamel} a contraction on a suitable ball in $C([0,T];\mathcal X_s)$ for $T>0$ small, proving local existence, uniqueness, and continuous dependence for \eqref{eq:exact-perturbation-system}.
\end{proof}

For the cutoff $\chi_R$ in the proof of Lemma~\ref{lem:profile-stability-space}, the estimate \eqref{eq:initial-cutoff-smallness} shows that localization is small in $\mathcal X_s$. The profile bounds and annular scaling give the estimate for the commutator $\|[\mathcal A_0,\chi_R]Q\|_{\dot H^k}\lesssim_{Q,s}R^{s_c-k}\to0$ for $k=s,k_\star$. Differentiating \eqref{eq:free-rescaled-propagator} on Schwartz data gives $C^\infty_{c,\mathrm{rad}}\subset D(\mathcal A_0)$. Also, the algebra estimate and \eqref{eq:initial-cutoff-smallness} imply $\chi_RN(Q)\to N(Q)$ in $\mathcal X_s$. Since $\chi_RQ\in D(\mathcal A_0)$ and
\begin{equation*}
	\mathcal A_0(\chi_RQ)
	=i\chi_RN(Q)+[\mathcal A_0,\chi_R]Q
	\longrightarrow iN(Q),
\end{equation*}
closedness of $\mathcal A_0$ yields $Q\in D(\mathcal A_0)$ and $\mathcal A_0Q=iN(Q)$. Thus the profile $Q$ is a mild solution of
\begin{equation}\label{eq:profile-mild-identity}
	Q=S_0(\tau)Q-i\int_0^\tau S_0(\tau-\sigma)N(Q)\,d\sigma.
\end{equation}
Consequently, subtracting $Q$ identifies a solution of the full nonlinear equation with the $\mathcal X_s$ perturbation.

\subsection{Spectral theory for the linearized operator}
\label{subsec:spectral-theory}
Our spectral analysis follows the strategy of obtaining dissipativity
modulo a compact perturbation; see
\cite{MerleRaphaelRodnianskiSzeftel2022,
MerleRaphaelRodnianskiSzeftel2022II,
BuckmasterCaoLaboraGomezSerrano2023}.
We represent the lower-order energy remainder by a compact operator
using the Riesz representation theorem, in a manner analogous to
the construction in \cite{chen2026vorticity}.

We use the matrix formulation and the associated
$\mathcal J$-invariance argument; see
\cite{Schlag2009StableManifolds,Li2023}.Fix $s\in(s_c,d/2)$ and set $\gamma_s=(s_c-s)/2<0$.
The spectral analysis takes place
on the complex Hilbert space
\begin{equation*}
	\mathbf X_s=\mathcal X_s\times\mathcal X_s,\qquad
	\|(f,g)\|_{\mathbf X_s}^2
	=\tfrac12\bigl(\|f\|_{\mathcal X_s}^2+\|g\|_{\mathcal X_s}^2\bigr).
\end{equation*}
Scalar perturbations are identified with the real-linear subspace
\begin{equation}\label{eq:matrix-real-subspace}
	\iota h=\begin{pmatrix}h\\\overline h\end{pmatrix},\qquad
	\mathbf X_{s,+}=\iota\mathcal X_s
	=\{\mathbf h\in\mathbf X_s:\mathcal J\mathbf h=\mathbf h\},\qquad
	\mathcal J\begin{pmatrix}f\\g\end{pmatrix}
	=\begin{pmatrix}\overline g\\\overline f\end{pmatrix}.
\end{equation}
Thus $\iota$ is a real-linear isometry and $\mathcal J$ is an antilinear isometric involution.

Let $Ch=\overline h$ and $\overline{\mathcal A_0}=C\mathcal A_0C =-i\Delta-\tfrac12 y\cdot\nabla_y-\frac{1}{p-1}-i\Omega$, with domain $CD(\mathcal A_0)$. The matrix linearized operator is
\begin{equation*}
	\mathbf A_Q=\mathbf A_0+\mathbf B,\qquad
	\mathbf A_0=
	\begin{pmatrix}\mathcal A_0&0\\0&\overline{\mathcal A_0}\end{pmatrix},
	\qquad
	\mathbf B=\begin{pmatrix}
		-iV_{\mathrm{diag}}          & -iV_{\mathrm{off}} \\
		i\overline{V_{\mathrm{off}}} & iV_{\mathrm{diag}}
	\end{pmatrix},
\end{equation*}
where the coefficients $V_{\mathrm{diag}}$, $V_{\mathrm{off}}$ are defined in \eqref{eq:linearized-coefficients}. Now we apply the bounded perturbation theorem as in the argument of Lemma~\ref{lem:linearized-flow} to the diagonal group
\begin{equation*}
	\mathbf S_0(\tau)=
	\begin{pmatrix}S_0(\tau)&0\\0&CS_0(\tau)C\end{pmatrix}.
\end{equation*}
Thus $\mathbf A_Q$, with domain $D(\mathcal A_0)\times CD(\mathcal A_0)$, generates a strongly continuous group $\mathbf S_Q(\tau)$. Direct computation gives
\begin{equation*}
	\mathcal J\mathbf A_Q=\mathbf A_Q\mathcal J,\qquad
	\mathbf A_Q\iota=\iota\mathcal A_Q,\qquad
	\mathbf S_Q(\tau)\iota=\iota S_Q(\tau).
\end{equation*}
The first identity holds on $D(\mathbf A_Q)$, which $\mathcal J$ preserves, and the second on $D(\mathcal A_Q)$; the group identity follows by uniqueness. In particular, the matrix evolution preserves $\mathbf X_{s,+}$ and agrees there with the scalar linearized evolution.
\begin{lem}\label{lem:minrev-energy}
	There is a bounded compact operator $\mathcal R$
	on $\mathbf X_s$ such that
	\begin{equation}\label{eq:minrev-energy}
		\operatorname{Re}\langle\mathbf A_Q\mathbf h,\mathbf h\rangle_{\mathbf X_s}
		\le \gamma_s\|\mathbf h\|_{\mathbf X_s}^2
		+\operatorname{Re}\langle\mathcal R\mathbf h,\mathbf h\rangle_{\mathbf X_s},
		\qquad \mathbf h\in D(\mathbf A_Q).
	\end{equation}
	The operator $\mathbf A_Q-\mathcal R$, with domain $D(\mathbf A_Q)$, generates a strongly continuous semigroup $\mathbf U$, and
	\begin{equation}\label{eq:minrev-corrected-decay}
		\|\mathbf U(\tau)\|_{\mathbf X_s\to\mathbf X_s}
		\le e^{\gamma_s\tau},\qquad \tau\ge0.
	\end{equation}
\end{lem}

\begin{proof}
	We first establish the compactness of the lower-order error terms $\mathbf B\mathbf h$ in the energy estimate. Let $\psi\in C_c^\infty(\mathbb R^d)$ be radial, equal to one on the unit ball, and supported in the ball of radius two. Let $h=h_{\mathrm{lo}}+h_{\mathrm{hi}}$, where $h_{\mathrm{lo}}=\mathcal{F}^{-1}\psi(|\xi|)\mathcal{F}h$ is the low frequency part of $h$. Since $s<d/2$, we have
	\begin{equation*}
		\|\partial^\beta h_{\mathrm{lo}}\|_\infty
		\lesssim_\beta
		\left(\int_{|\xi|\le2}|\xi|^{2|\beta|-2s}\,d\xi\right)^{1/2}
		\|h\|_{\dot H^s}
		\lesssim_\beta\|h\|_{\mathcal X_s}.
	\end{equation*}
	For the high-frequency part, we have $\|h_{\mathrm{hi}}\|_{H^{k_\star}}\lesssim\|h\|_{\dot H^{k_\star}}$. These estimates give the localization bound below. Let $\chi_L(y)=\chi(y/L)$, where $\chi\in C_c^\infty(B_2)$ is radial and equals one on $B_1$. For each fixed $L$, localization maps $\mathcal X_s$ boundedly into $H^{k_\star}$. Hence, choosing $\max\{s,d/2\}<r<k_\star$, Rellich compactness shows that $h\mapsto\chi_Lh$ is compact into $H^r$. Since multiplication by any of the factors $u\in\{V_{\mathrm{diag}},V_{\mathrm{off}},\overline{V_{\mathrm{off}}}\}$ is bounded on $H^r$ and $(\chi_Lu)h=u(\chi_Lh)$, the corresponding localized multiplication operators are compact from $\mathcal X_s$ into $\dot H^s$.

	For the tails, \eqref{eq:linearized-coefficient-derivative-bounds} and the same dyadic scaling argument as in Lemma~\ref{lem:profile-stability-space} give, for $u\in\{V_{\mathrm{diag}},V_{\mathrm{off}},\overline{V_{\mathrm{off}}}\}$,
	\begin{equation*}
		\|(1-\chi_L)u\|_{\mathcal X_s}
		\lesssim L^{\frac d2-2-s}
		\lesssim L^{s_c-s}.
	\end{equation*}
	Together with the algebra estimate \eqref{eq:algebra-estimate}, this gives
	\begin{equation*}
		\|(1-\chi_L)V_{\mathrm{diag}}h\|_{\mathcal{X}_s}
		+\|(1-\chi_L)V_{\mathrm{off}}h\|_{\mathcal{X}_s}
		\lesssim L^{s_c-s}\|h\|_{\mathcal X_s}.
	\end{equation*}

	Since $s>s_c$, the right-hand coefficient tends to zero. Therefore
	\begin{equation*}
		\mathbf h\longmapsto |D|^s\mathbf B\mathbf h
		\quad\text{is compact from }\mathbf X_s\text{ into }L^2(\mathbb R^d;\mathbb C^2).
	\end{equation*}

	Now let
	\begin{equation*}
		\mathcal C\mathbf h=[\nabla^{k_\star},\mathbf B]\mathbf h.
	\end{equation*}
	Each component is a finite sum of terms $(\partial^\alpha u)\partial^\beta h$, with $u\in\{V_{\mathrm{diag}},V_{\mathrm{off}},\overline{V_{\mathrm{off}}}\}$ where $|\alpha|\ge1$ and $|\alpha|+|\beta|=k_\star$. The localized term $\chi_L(\partial^\alpha u)\partial^\beta h$ is compact into $L^2$ by local $H^{k_\star}$ control, Rellich compactness, and $|\beta|\le k_\star-1$. For its tail, \eqref{eq:slow-profile-symbol-bounds} gives:
	\begin{align*}
		\|(1-\chi_L)(\partial^\alpha u)\partial^\beta h_{\mathrm{lo}}\|_{L^{2}}\lesssim \|(1-\chi_L)(\partial^\alpha u)\|_{L^{2}}\|\partial^\beta h_{\mathrm{lo}}\|_{L^{\infty}}
		 & \lesssim L^{d/2-2-|\alpha|}\|h\|_{\mathcal X_s}, \\
		\|(1-\chi_L)(\partial^\alpha u)\partial^\beta h_{\mathrm{hi}}\|_{L^{2}}\lesssim \|(1-\chi_L)(\partial^\alpha u)\|_{L^{\infty}} \|\partial^\beta h_{\mathrm{hi}}\|_{L^{2}}
		 & \lesssim L^{-2-|\alpha|}\|h\|_{\mathcal X_s}.
	\end{align*}
	Both exponents are negative because $d\le5$ and $|\alpha|\ge1$. Thus $\mathcal C$ is compact from $\mathbf X_s$ into $L^2(\mathbb R^d;\mathbb C^2)$ as well. It follows that
	\begin{equation*}
		\mathcal K\mathbf h=(|D|^s\mathbf B\mathbf h,\mathcal C\mathbf h)
		:\mathbf X_s\longrightarrow\mathcal Y
	\end{equation*}
	is compact, where $\mathcal Y$ is the Hilbert direct sum of these two $L^2$ spaces, with $\|(f,g)\|_{\mathcal Y}^2=\frac{1}{2}(\|f\|_{L^2}^2+\|g\|_{L^2}^2)$.

	Next we work on the energy estimate. The calculations are first made on radial Schwartz pairs. Differentiating the free-group norm identity following \eqref{eq:free-rescaled-propagator} gives
	\begin{equation*}
		\operatorname{Re}\langle |D|^r\mathbf A_0\mathbf h,
		|D|^r\mathbf h\rangle_{L^2}
		=\frac{s_c-r}{2}\|\mathbf h\|_{\dot H^r}^2,
		\qquad r=s,k_\star.
	\end{equation*}
	At the top order,
	\begin{equation*}
		\nabla^{k_\star}(\mathbf B\mathbf h)
		=\mathbf B\nabla^{k_\star}\mathbf h+\mathcal C\mathbf h.
	\end{equation*}
	Because $V_{\mathrm{diag}}$ is real, the diagonal terms of $\mathbf B$ make zero real contribution. For any pair of tensors $\mathbf w$,
	\begin{equation*}
		\operatorname{Re}\langle\mathbf B\mathbf w,\mathbf w\rangle_{L^2}
		\le\|V_{\mathrm{off}}\|_\infty\|\mathbf w\|_{L^2}^2.
	\end{equation*}
	Our choice of $k_\star$ implies
	\begin{equation*}
		\frac{s_c-k_\star}{2}+\|V_{\mathrm{off}}\|_\infty
		\le\frac{s_c-s}{2}=\gamma_s.
	\end{equation*}
	Define the bounded isometry
	\begin{equation*}
		\mathcal E\mathbf h=(|D|^s\mathbf h,\nabla^{k_\star}\mathbf h)
		:\mathbf X_s\longrightarrow\mathcal Y,
		\qquad \|\mathcal E\mathbf h\|_{\mathcal Y}=\|\mathbf h\|_{\mathbf X_s}.
	\end{equation*}
	Combining the two derivative levels gives
	\begin{equation}\label{eq:AQ-energy-estimate}
		\operatorname{Re}\langle\mathbf A_Q\mathbf h,\mathbf h\rangle_{\mathbf X_s}
		\le\gamma_s\|\mathbf h\|_{\mathbf X_s}^2
		+\operatorname{Re}\langle\mathcal K\mathbf h,
		\mathcal E\mathbf h\rangle_{\mathcal Y}.
	\end{equation}
	The free energy identity holds for every $\mathbf h\in D(\mathbf A_0)$. The estimate for the bounded perturbation $\mathbf B$ extends from radial Schwartz pairs to $\mathbf X_s$ by density. Since $D(\mathbf A_Q)=D(\mathbf A_0)$, their combination proves \eqref{eq:AQ-energy-estimate} on $D(\mathbf A_Q)$.

		Use inner products linear in the first argument. The sesquilinear form
		\begin{equation*}
			\mathfrak k(\mathbf h,\mathbf g)
			=\langle\mathcal K\mathbf h,\mathcal E\mathbf g\rangle_{\mathcal Y}
		\end{equation*}
		is bounded, since $|\mathfrak k(\mathbf h,\mathbf g)| \le\|\mathcal K\|\,\|\mathbf h\|_{\mathbf X_s}\|\mathbf g\|_{\mathbf X_s}$. By Riesz representation, there is a unique bounded complex-linear operator $\mathcal R:\mathbf X_s\to\mathbf X_s$ satisfying
		\begin{equation*}
			\langle\mathcal R\mathbf h,\mathbf g\rangle_{\mathbf X_s}
			=\mathfrak k(\mathbf h,\mathbf g),
			\qquad \mathcal R=\mathcal E^*\mathcal K.
		\end{equation*}

		The operator $\mathcal R$ is compact because $\mathcal K$ is compact and $\mathcal E^*$ is bounded. In particular,
		\begin{equation*}
			\operatorname{Re}\langle\mathcal R\mathbf h,\mathbf h\rangle_{\mathbf X_s}
			=\operatorname{Re}\langle\mathcal K\mathbf h,
			\mathcal E\mathbf h\rangle_{\mathcal Y}.
		\end{equation*}
		Thus the preceding energy estimate proves \eqref{eq:minrev-energy} directly.

		Since $\mathbf A_Q$ already generates $\mathbf S_Q$, bounded perturbation by $-\mathcal R$ gives the auxiliary group $\mathbf U$, with generator $\mathbf A_Q-\mathcal R$ and domain $D(\mathbf A_Q)$. This yields
	\begin{equation*}
		\operatorname{Re}\langle(\mathbf A_Q-\mathcal R)\mathbf h,\mathbf h\rangle_{\mathbf X_s}
		\le\gamma_s\|\mathbf h\|_{\mathbf X_s}^2,
		\qquad \mathbf h\in D(\mathbf A_Q).
	\end{equation*}
	Then
	\begin{equation*}
		\frac12\frac{d}{d\tau}\|\mathbf U(\tau)\mathbf h\|_{\mathbf X_s}^2
		=\operatorname{Re}\langle(\mathbf A_Q-\mathcal R)\mathbf U(\tau)\mathbf h,
		\mathbf U(\tau)\mathbf h\rangle_{\mathbf X_s}
		\le\gamma_s\|\mathbf U(\tau)\mathbf h\|_{\mathbf X_s}^2.
	\end{equation*}
	Gronwall's inequality and density prove \eqref{eq:minrev-corrected-decay}.

\end{proof}

We recall the following theorem.
\begin{thm}[{\cite[Corollary~IV.2.11]{EngelNagel2000}}]
	\label{cor:semigroup-growth-bound}
	Suppose that $A:D(A)\subset X\to X$ generates a strongly continuous
	semigroup $T=(T(t))_{t\geq 0}$ on a complex Banach space $X$.
	Denote its growth bound and essential growth bound by
	$\omega_0(T)$ and $\omega_{\mathrm{ess}}(T)$:
	\begin{align*}
		\omega_0(T)              & =\inf\{w\in\mathbb R:\exists M_w<\infty,\
		\|T(\tau)\|\le M_we^{w\tau}\text{ for all }\tau\ge0\},               \\
		\omega_{\mathrm{ess}}(T) & =\inf\{w\in\mathbb R:\exists M_w<\infty,\
		\|T(\tau)\|_{\mathrm{ess}}\le M_we^{w\tau}\text{ for all }\tau\ge0\},
	\end{align*}
	where $\|L\|_{\mathrm{ess}}=\inf\{\|L-K\|:K\text{ compact}\}$ and $\sigma(A)$ denotes the full spectrum of $A$. Then
	\begin{equation*}
		\omega_0(T)
		=\max\left\{\omega_{\mathrm{ess}}(T),
		\sup_{\lambda\in\sigma(A)}\operatorname{Re}\lambda\right\}.
	\end{equation*}
	For each real number $w>\omega_{\mathrm{ess}}(T)$, the set
	\begin{equation*}
		\sigma(A)\cap
		\bigl\{\lambda\in\mathbb{C}:
		\operatorname{Re}\lambda\geq w\bigr\}
	\end{equation*}
	is finite, and its associated spectral projection has finite rank. In particular, its elements are isolated eigenvalues of $A$ with finite algebraic multiplicity.
\end{thm}

By \cite[Proposition~IV.2.12]{EngelNagel2000}, compact perturbations preserve the essential growth bound, so
\begin{equation}\label{eq:minrev-essential-growth}
	\omega_{\mathrm{ess}}(\mathbf S_Q)
	=\omega_{\mathrm{ess}}(\mathbf U)
	\le \omega_0(\mathbf U)\le\gamma_s<0.
\end{equation}
Then in our case Corollary~\ref{cor:semigroup-growth-bound} gives the following proposition.
\begin{prop}\label{prop:minrev-finite-spectrum}
	For every $\delta>0$,
	\begin{equation*}
		\sigma(\mathbf A_Q)\cap\{\lambda:\operatorname{Re}\lambda\ge\gamma_s+\delta\}
	\end{equation*}
	is a finite set of isolated eigenvalues of finite algebraic multiplicity. In particular, the sum of the generalized eigenspaces with $\operatorname{Re}\lambda\ge0$ is finite-dimensional.
\end{prop}

We now introduce spectral projections to separate the stable part of the linearized evolution from the finite-dimensional neutral and unstable parts. We denote the stable and unstable projections by $P_{\mathrm{st}}$ and $P_{\mathrm{un}}$, respectively.
\begin{prop}\label{prop:minrev-splitting}
	There are bounded real-linear projections $P_{\mathrm{un}}$ and
	$P_{\mathrm{st}}=I-P_{\mathrm{un}}$, commuting with $\mathcal A_Q$ and $S_Q$, such
	that $\operatorname{Ran}P_{\mathrm{un}}$ is finite-dimensional. For some
	$\omega>0$ and $C<\infty$,
	\begin{equation}\label{eq:minrev-stable}
		\|S_Q(\tau)P_{\mathrm{st}}\|_{\mathcal X_s\to \mathcal X_s}
		\le Ce^{-\omega\tau},\qquad \tau\ge0.
	\end{equation}
	For every $\eta>0$, we also have
	\begin{equation*}
		\|S_Q(-\tau)P_{\mathrm{un}}\|_{\mathcal X_s\to \mathcal X_s}
		\le C_\eta e^{\eta\tau},\qquad \tau\ge0.
	\end{equation*}
\end{prop}
\begin{proof}

	By Proposition~\ref{prop:minrev-finite-spectrum}, there exists $\gamma_{s}<-\omega<0$ such that the set
	\begin{equation*}
		\Lambda_{\mathrm{un}}=\sigma(\mathbf A_Q)\cap\{\lambda:\operatorname{Re}\lambda\ge -\omega\}
		=\sigma(\mathbf A_Q)\cap\{\lambda:\operatorname{Re}\lambda\ge0\}
	\end{equation*}
	is finite and consists of isolated eigenvalues of finite algebraic multiplicity. Let $\mathbf P_{\mathrm{un}}$ be the spectral projection associated with $\Lambda_{\mathrm{un}}$, and set $\mathbf P_{\mathrm{st}}=I-\mathbf P_{\mathrm{un}}$. The standard spectral argument (e.g., \cite[Lemma~7.11]{BuckmasterCaoLaboraGomezSerrano2025}) gives the closed invariant decomposition
	\begin{equation*}
		\mathbf X_s=\operatorname{Ran}\mathbf P_{\mathrm{un}}
		\oplus\operatorname{Ran}\mathbf P_{\mathrm{st}}.
	\end{equation*}
	From Corollary~\ref{cor:semigroup-growth-bound} and \eqref{eq:minrev-essential-growth}, we have for $\tau\geq 0$,
	\begin{equation*}
		\|\mathbf{S}_{Q}(\tau)\mathbf{P}_{\mathrm{st}}\|\leq Ce^{-\omega \tau}.
	\end{equation*}
	From the definition of $\Lambda_{\mathrm{un}}$, all eigenvalues have nonnegative real part. Jordan normal form therefore gives, for all $\eta >0$ and $\tau\geq 0$,
	\begin{equation*}
		\|\mathbf{S}_{Q}(-\tau)\mathbf{P}_{\mathrm{un}}\|\leq C_{\eta}e^{\eta \tau}.
	\end{equation*}
	Moreover, since $\mathcal J$ from \eqref{eq:matrix-real-subspace} commutes with $\mathbf A_Q$, we have
	\begin{equation*}
		\mathcal{J}(z\mathbf{I}-\mathbf{A}_{Q})^{-1}\mathcal{J}=(\bar{z}\mathbf{I}-\mathbf{A}_{Q})^{-1}.
	\end{equation*}
	Let $\mathbf{P}_{\lambda}$ be the Riesz projection for $\lambda$. Then
	\begin{equation*}
		\mathcal{J}\mathbf{P}_{\lambda}\mathcal{J}=\mathbf{P}_{\bar{\lambda}}.
	\end{equation*}
	Since $\Lambda_{\mathrm{un}}$ is invariant under conjugation, we have $\mathcal{J}\mathbf{P}_{\mathrm{un}}\mathcal{J}=\mathbf{P}_{\mathrm{un}}$. Thus $\mathcal{J}$ commutes with the projection $\mathbf{P}_{\mathrm{un}}$. Hence $\mathbf{P}_{\mathrm{un}}$ preserves $\mathbf X_{s,+}$ and induces bounded real-linear projections
	\begin{equation*}
		P_\alpha=\iota^{-1}\mathbf P_\alpha\iota,
		\qquad\alpha\in\{\mathrm{un},\mathrm{st}\},
	\end{equation*}
	with the required commutation properties.
\end{proof}

We record the two spectral directions induced by the symmetries of \eqref{eq:nls}. Phase invariance and time translation give, respectively,
\begin{equation*}
	\mathcal A_Q(iQ)=0,
	\qquad
	\mathcal A_Q\left(-\frac{1}{p-1}Q+i\Omega Q
	-\frac12 y\cdot\nabla Q\right)
	=-\frac{1}{p-1}Q+i\Omega Q-\frac12 y\cdot\nabla Q.
\end{equation*}
Indeed, the first identity follows by differentiating the constant phase, and the second one follows by differentiating the blow-up time. They can also be checked directly from the profile equation. The symbol bounds of Lemma~\ref{lem:slow-profile-symbol-bounds} show that both eigenfunctions belong to $D(\mathcal A_Q)$. Hence both directions lie in $\operatorname{Ran}P_{\mathrm{un}}$, and thus this space has dimension at least two. The construction below selects all neutral and unstable components at once, and therefore requires no separate modulation of the blow-up time or of the constant phase. We do not determine whether the two symmetry modes exhaust this space.

\subsection{Stability and localization}
\label{subsec:global-stability}

The following proposition gives a finite-codimensional stability result for the linearized evolution. This result allows us to construct global solutions to the nonlinear equation.
\begin{prop}
	\label{prop:global-stability}
	For the fixed $s$ and $\omega$ above, and every $0<\mu<\omega$,
	there is $\epsilon_0>0$ such that, for every
	$\xi\in P_{\mathrm{st}}\mathcal X_s$ with
	$\|\xi\|_{\mathcal X_s}<\epsilon_0$, the perturbation equation
	\eqref{eq:exact-perturbation-system} admits a global solution $v$ with
	\begin{equation*}
		v(0)=\xi-\int_0^\infty
		S_Q(-\sigma)P_{\mathrm{un}}\mathcal F(v(\sigma))\,d\sigma
	\end{equation*}
	and
	\begin{equation}\label{eq:stationary-stability-decay}
		\sup_{\tau\geq0}e^{\mu\tau}\|v(\tau)\|_{\mathcal X_s}
		\lesssim\|\xi\|_{\mathcal X_s}.
	\end{equation}
	Moreover, for every sufficiently large $R$, such data can be chosen with $Q+v(0)=\chi_RQ+g_R$, where $g_R$ lies in a fixed finite-dimensional subspace of radial $C_c^\infty$ and tends to zero in every $H^m$ as $R\to\infty$.
\end{prop}
\begin{proof}
	To establish the decay estimate
\eqref{eq:stationary-stability-decay}, we follow the classical
Lyapunov--Perron construction; see, for example,
\cite[Section 9.2]{Teschl2012}. We then obtain smooth compactly
supported initial data by a finite-dimensional localization
argument related to the constructions in
\cite{LiZhou2025,Li2023}.

We prescribe the stable component $\xi$ and choose the unstable component to cancel the backward evolution of the nonlinear term.

	For $\xi\in P_{\mathrm{st}}\mathcal X_s$, solve the equation
	\begin{equation}\label{eq:stationary-lyapunov-perron}
		\begin{aligned}
			v(\tau)={} & S_Q(\tau)\xi
			+\int_0^\tau S_Q(\tau-\sigma)P_{\mathrm{st}}\mathcal F(v(\sigma))\,d\sigma
			-\int_\tau^\infty
			S_Q(\tau-\sigma)P_{\mathrm{un}}\mathcal F(v(\sigma))\,d\sigma.
		\end{aligned}
	\end{equation}
	Choose $0<\eta<2\mu$ in Proposition~\ref{prop:minrev-splitting}. We use the norm $\|v\|_\mu:=\sup_{\tau\geq0}e^{\mu\tau}\|v(\tau)\|_{\mathcal X_s}$. For $\|v\|_\mu\leq1$, \eqref{eq:stability-nonlinear-estimate} with $w=0$ gives, since $\mathcal F(0)=0$,
	\begin{equation*}
		\|\mathcal F(v(\sigma))\|_{\mathcal X_s}
		\leq C_{Q,s}\|v(\sigma)\|_{\mathcal X_s}^2
		\leq C_{Q,s}e^{-2\mu\sigma}\|v\|_\mu^2.
	\end{equation*}
	Here we used $\|v(\sigma)\|_{\mathcal X_s}\leq e^{-\mu\sigma}\|v\|_\mu\leq1$. Hence, for every $\tau\geq0$, we have estimates
	\begin{equation*}
		\begin{aligned}
			e^{\mu\tau}\left\|\int_0^\tau
			S_Q(\tau-\sigma)P_{\mathrm{st}}\mathcal F(v(\sigma))\,d\sigma
			\right\|_{\mathcal X_s}
			 & \leq CC_{Q,s}\|v\|_\mu^2
			\int_0^\tau e^{\mu\tau}e^{-\omega(\tau-\sigma)}e^{-2\mu\sigma}\,d\sigma   \\
			 & \leq\frac{CC_{Q,s}}{\omega-\mu}\|v\|_\mu^2,                            \\
			e^{\mu\tau}\left\|\int_\tau^\infty
			S_Q(\tau-\sigma)P_{\mathrm{un}}\mathcal F(v(\sigma))\,d\sigma
			\right\|_{\mathcal X_s}
			 & \leq C_\eta C_{Q,s}\|v\|_\mu^2
			\int_\tau^\infty e^{\mu\tau}e^{\eta(\sigma-\tau)}e^{-2\mu\sigma}\,d\sigma \\
			 & \leq\frac{C_\eta C_{Q,s}}{2\mu-\eta}\|v\|_\mu^2.
		\end{aligned}
	\end{equation*}
	For $\|v\|_\mu,\|w\|_\mu\leq1$, the same bounds hold with $\mathcal F(v)-\mathcal F(w)$ in the integrands and $\|v\|_\mu^2$ replaced by $(\|v\|_\mu+\|w\|_\mu)\|v-w\|_\mu$. Together with \eqref{eq:minrev-stable} and \eqref{eq:stability-nonlinear-estimate}, this makes the right-hand side of the equation \eqref{eq:stationary-lyapunov-perron} a contraction on a ball of radius $C\|\xi\|_{\mathcal X_s}$ in
	\begin{equation*}
		\left\{v\in C([0,\infty);\mathcal X_s):
		\sup_{\tau\geq0}e^{\mu\tau}\|v(\tau)\|_{\mathcal X_s}<\infty\right\}.
	\end{equation*}
	The contraction mapping theorem therefore gives a fixed point $v$ on $[0,\infty)$. Subtracting $S_Q(\tau)$ times \eqref{eq:stationary-lyapunov-perron} at $\tau=0$ from the same equation at time $\tau$ yields
	\begin{equation*}
		v(\tau)=S_Q(\tau)v(0)
		+\int_0^\tau S_Q(\tau-\sigma)\mathcal F(v(\sigma))\,d\sigma.
	\end{equation*}
	This shows the existence of a global solution to \eqref{eq:exact-perturbation-system}, with the decay estimate \eqref{eq:stationary-stability-decay}.

	Now we work on the localization. For each $\xi\in B_{\epsilon_0}(0)\cap P_{\mathrm{st}}\mathcal X_s$, let $v_\xi$ denote the unique fixed point constructed above for that stable datum. Define the map $\Phi:B_{\epsilon_0}(0)\cap P_{\mathrm{st}}\mathcal X_s \to P_{\mathrm{un}}\mathcal X_s$ by
	\begin{equation*}
		\Phi(\xi):=P_{\mathrm{un}}v_\xi(0)
		=-\int_0^\infty S_Q(-\sigma)P_{\mathrm{un}}
		\mathcal F(v_\xi(\sigma))\,d\sigma.
	\end{equation*}
	The integral converges in $\mathcal X_s$ by the preceding estimates, and $v_\xi(0)=\xi+\Phi(\xi)$. In particular, $\Phi(0)=0$ since $v_0=0$. The same estimates for differences give
	\begin{equation}\label{eq:unstable-data-map}
		\|\Phi(\xi)-\Phi(\eta)\|_{\mathcal X_s}
		\lesssim
		\bigl(\|\xi\|_{\mathcal X_s}+\|\eta\|_{\mathcal X_s}\bigr)
		\|\xi-\eta\|_{\mathcal X_s}.
	\end{equation}

	To localize, choose a finite-dimensional real space $F\subset C^\infty_{c,\mathrm{rad}}$ for which $P_{\mathrm{un}}|_F:F\to P_{\mathrm{un}}\mathcal X_s$ is an isomorphism. Density of $C^\infty_{c,\mathrm{rad}}$ in $\mathcal X_s$ ensures such a choice. By \eqref{eq:initial-cutoff-smallness}, $(\chi_R-1)Q\to0$ in $\mathcal X_s$. Now we show that the equation
	\begin{equation}\label{eq:localized-data-selection}
		g_R=(P_{\mathrm{un}}|_F)^{-1}
		\left[\Phi\bigl(P_{\mathrm{st}}((\chi_R-1)Q+g_R)\bigr)
		-P_{\mathrm{un}}((\chi_R-1)Q)\right]
	\end{equation}
	is a contraction on a ball in $F$ of radius $C\|(\chi_R-1)Q\|_{\mathcal X_s}$ when $R$ is sufficiently large. Denote the right-hand side of \eqref{eq:localized-data-selection} by $T_R(g)$. By \eqref{eq:unstable-data-map}, $\Phi(0)=0$, and boundedness of the projections and $(P_{\mathrm{un}}|_F)^{-1}$, for $g,h$ in this ball we have
	\begin{equation*}
		\begin{aligned}
			\|T_R(g)\|_{\mathcal X_s} & \leq K\|(\chi_R-1)Q\|_{\mathcal X_s}
			+K(1+C)^2\|(\chi_R-1)Q\|_{\mathcal X_s}^2,                                                  \\
			\|T_R(g)-T_R(h)\|_{\mathcal X_s}
			                          & \leq K(1+C)\|(\chi_R-1)Q\|_{\mathcal X_s}\|g-h\|_{\mathcal X_s}
		\end{aligned},
	\end{equation*}
	with $K$ independent of $C$ and $R$. Fixing $C>2K$ and then taking $R$ large makes $T_R$ preserve this ball and have Lipschitz constant less than one; its arguments also lie in the domain of $\Phi$. Its solution satisfies $P_{\mathrm{un}}((\chi_R-1)Q+g_R)=\Phi(P_{\mathrm{st}}((\chi_R-1)Q+g_R))$, so $v(0)=(\chi_R-1)Q+g_R$ lies on the constructed graph and $Q+v(0)=\chi_RQ+g_R\in C_c^\infty$. Since $F$ is fixed and finite-dimensional, $g_R\to0$ in every $H^m$.
\end{proof}

\begin{proof}[Proof of Theorem~\ref{thm:blowup}]
	Fix $0<\mu<\omega$ and choose localized data from Proposition~\ref{prop:global-stability}. By \eqref{eq:stability-Linfty-embedding} and \eqref{eq:stationary-stability-decay}, $V=Q+v$ is uniformly bounded and
	\begin{equation*}
		\|V(\tau)-Q\|_{L^\infty}\lesssim e^{-\mu\tau}.
	\end{equation*}
	Its initial datum is smooth and compactly supported. Rewriting the Duhamel formula for $v$ relative to $S_0$ and using \eqref{eq:profile-mild-identity} shows that $V$ solves \eqref{eq:complex-rescaled-evolution}. Higher regularity persists, and the tame product estimate yields, for every integer $m\geq3$ for which $V(0)\in H^m$,
	\begin{equation}\label{eq:total-field-Hm-bound}
		\|V(\tau)\|_{H^m}
		\leq\|V(0)\|_{H^m}
		\exp\left(\left(\frac d4-\frac{1}{p-1}\right)\tau+
		C_m\int_0^\tau\|V(\sigma)\|_\infty^{p-1}\,d\sigma\right).
	\end{equation}
	Consequently, a bounded $L^\infty$ norm prevents finite-time breakdown and propagates all initial Sobolev regularity. Uniqueness and \eqref{eq:total-field-Hm-bound} therefore give $V\in C([0,\infty);H^m)$ for every integer $m\geq0$. For any $T>0$, undoing \eqref{eq:coordinates} gives
	\begin{equation*}
		u(t,x)=(T-t)^{-\frac{1}{p-1}}e^{-i\Omega\tau}V(\tau,y).
	\end{equation*}
	Therefore,
	\begin{equation*}
		\left\|(T-t)^{\frac{1}{p-1}}e^{i\Omega\tau}
		u\left(t,\sqrt{T-t}\,\cdot\right)-Q\right\|_\infty
		=O((T-t)^\mu),
	\end{equation*}
	and hence
	\begin{equation*}
		\|u(t)\|_\infty=(T-t)^{-\frac{1}{p-1}}
		\left(\|Q\|_\infty+O((T-t)^\mu)\right).
	\end{equation*}
	The resulting solution lies in $C([0,T);H^m)$ for every $m$, with smooth compactly supported initial data. Since $Q(0)=A>0$, we have $\|Q\|_\infty>0$, and $T$ is its maximal time of existence. Finally, \eqref{eq:slow-profile-symbol-bounds} gives $|Q(r)|\gtrsim r^{-2/(p-1)}$ for $r\geq1$. Therefore,
	\begin{equation*}
		\int_{\mathbb R^d}|Q(x)|^{\frac{d(p-1)}{2}}\,dx
		\gtrsim \int_1^\infty r^{d-1}r^{-d}\,dr
		=\int_1^\infty \frac{dr}{r}=\infty,
	\end{equation*}
	so $Q\notin L^{d(p-1)/2}$. As $V(\tau)\to Q$ pointwise, Fatou's lemma gives $\|V(\tau)\|_{L^{d(p-1)/2}}\to\infty$. Sobolev embedding and scaling invariance yield
	\begin{equation*}
		\|u(t)\|_{\dot H^{s_c}}
		\gtrsim\|u(t)\|_{L^{\frac{d(p-1)}{2}}}
		=\|V(\tau)\|_{L^{\frac{d(p-1)}{2}}}\longrightarrow\infty,
	\end{equation*}
	which completes the proof.
\end{proof}

\begingroup
\hbadness=10000
\bibliographystyle{plain}
\bibliography{references}
\endgroup

\begin{tabular}{l}
	\textbf{Kevin Buck}                          \\
	{Department of Mathematics}                  \\
	{Brown University}                           \\
	{151 Thayer Street}                          \\
	{Providence, RI 02912, USA}                  \\
	{Email: kevin\_buck@brown.edu}               \\ \\
	\textbf{Alvaro Carballeira}                  \\
	{Department of Mathematics}                  \\
	{Brown University}                           \\
	{151 Thayer Street}                          \\
	{Providence, RI 02912, USA}                  \\
	{Email: alvaro\_carballeira\_mora@brown.edu} \\ \\
	\textbf{Javier G\'omez-Serrano}              \\
	{Department of Mathematics}                  \\
	{Brown University}                           \\
	{151 Thayer Street}                          \\
	{Providence, RI 02912, USA}                  \\
	{Email: javier\_gomez\_serrano@brown.edu}    \\ \\
	\textbf{Jia Shi}                             \\
	{Department of Mathematics}                  \\
	{Indiana University Bloomington}             \\
	{831 E. 3rd St.}                             \\
	{Bloomington, IN 47405, USA}                 \\
	{Email: js289@iu.edu}                        \\
\end{tabular}

\end{document}